\documentclass[12pt]{amsart}
\usepackage{amsmath,amsfonts,amssymb,geometry}
\usepackage{graphicx,amsmath, amscd}
\usepackage{color,tikz}
\usepackage{url}

\def\be{\begin{equation}}
\def\ee{\end{equation}}
\def\bes{\begin{equation*}}
\def\ees{\end{equation*}}

\def\bea{\begin{eqnarray}}
\def\eea{\end{eqnarray}}
\def\bt{\begin{theorem}}
\def\et{\end{theorem}}
\def\bl{\begin{lemma}}
\def\el{\end{lemma}}
\def\br{\begin{remark}}
\def\er{\end{remark}}
\def\bc{\begin{corollary}}
\def\ec{\end{corollary}}
\def\bd{\begin{definition}}
\def\ed{\end{definition}}
\def\eprop{\end{proposition}}
\def\bprop{\begin{proposition}}

\newcommand{\E}{\mathbb E}

\newcommand{\ip}[2]{\left\langle#1,#2\right\rangle}

\def\otoh{On the other hand, }

\def\a{\alpha}

\def\bbR{\mathbb{R}}

\def\ep{\varepsilon}

\def\b1{B_{1}^}

\def\ba{\begin{array}}
\def\ea{\end{array}}
\def\ben{\begin{enumerate}}
\def\een{\end{enumerate}}

\theoremstyle{plain}

\newtheorem{lemma}{Lemma}
\newtheorem{theorem}[lemma]{Theorem}
\newtheorem*{remark}{Remark}
\newtheorem{proposition}[lemma]{Proposition}
\newtheorem{fact}[lemma]{Fact}
\newtheorem{corollary}[lemma]{Corollary}

\newtheorem*{definition}{Definition}

\renewcommand{\leq}{\leqslant}
\renewcommand{\geq}{\geqslant}

\newcommand{\R}{\mathbb{R}}

\renewcommand{\P}{\mathbb{P}}

\begin{document}
\title[Optimal constants in concentration inequalities]{Optimal constants in concentration inequalities on the sphere and in the Gauss space}
\author{Guillaume Aubrun}
\address{{Institut Camille Jordan, Universit\'{e} Claude Bernard Lyon 1, 43 boulevard du 11 novembre 1918, 69622 Villeurbanne cedex, France}}
\address{{Univ.\ Lyon, ENS Lyon, UCBL, CNRS, Inria, LIP, F-69342, Lyon Cedex 07,
France.}}
\email{aubrun@math.univ-lyon1.fr}
\author{Justin Jenkinson}
\address{{Case Western Reserve University, Cleveland, Ohio 44106-7058, U.S.A.}}
\address{Current address: University of Akron, Akron, OH 44325-4002, U.S.A.}
\email{jjenkins2@uakron.edu}

\author{Stanislaw J. Szarek}
\address{{Case Western Reserve University, Cleveland, Ohio 44106-7058, U.S.A.}}
\address{{Institut de Math\'ematiques de Jussieu, 
Sorbonne Universit\'e, 4 place Jussieu, 75252 Paris cedex 05, France}}
\email{szarek@cwru.edu}


\begin{abstract}
We show several variants of concentration inequalities on the sphere stated as subgaussian estimates with optimal constants. 
For a Lipschitz function, we give one-sided and two-sided bounds for deviation from the median as well as from the mean. For example, we show that if $\mu$ is the normalized surface measure on $S^{n-1}$ with $n\geq 3$,   $f : S^{n-1} \to \R$ is $1$-Lipschitz, $M$ is the median of $f$, and $t >0$,  then $\mu\big(f \geq M +t\big) \leq \frac 12 e^{-nt^2/2}$. 
If $M$ is the mean of $f$, we have a two-sided bound $\mu\big(|f - M| \geq t\big) \leq e^{-nt^2/2}$. Consequently, if $\gamma$ is the standard  Gaussian measure on $\R^n$ and $f : \R^{n} \to \R$ (again, $1$-Lipschitz, with the mean equal to $M$), then 
$\gamma \big(|f - M| \geq t\big) \leq  e^{-t^2/2}$. 
These bounds are slightly better and arguably more elegant than those available elsewhere in the literature. 
\end{abstract}

\maketitle

\section{Introduction and the main results} L\'evy's isoperimetric inequality on the sphere in $\R^n$ 
\cite{Levy1951, Schmidt} is one of the most useful tools in the study of high-dimensional phenomena.  
The isoperimetric inequality itself is a very precise result: {\em Among 
the subsets of  the sphere of given measure, the caps have the smallest 
boundary}.  On the other hand, typical applications appeal to its corollaries, 
which  exhibit varying degrees of tightness.  
Such corollaries are most often expressed as subgaussian concentration inequalities of the form 
 \begin{equation} \mu(\left\{f \geq M+t\right\}) \leq  C\ e^{-cn t ^2} ,
 \label{eq:normal}
\end{equation}
valid for any $1$-Lipschitz real valued function on the sphere 
and any $t>0$, 
where $\mu =\mu_n$ is the normalized surface measure on the sphere, 
$M$ is either the median or the mean of~$f$, and $C,c>0$ are 
(explicit or effectively computable) constants, independent of  
$n$, $f$,  and~$t$. 
Perhaps the most frequently cited variant of a spherical concentration inequality 
comes from the  influential 1986 book of Milman and Schechtman \cite{MS86}. 

\begin{fact}\label{fact:MS} If $f: S^{n+1} \rightarrow \R$ is a $1$-Lipschitz
function  (with respect to the geodesic distance) and $M$ is its median, then for every $t> 0$
\begin{equation} \label{eq:MS} 
 \mu(\left\{f \geq M+t\right\}) \leq  \sqrt{{\pi}/8} \ e^{-n t ^2/2}.
 \end{equation}
Equivalently, if $A \subset S^{n+1}$ is such that $\mu(A)\geq \frac 12$, then -- for every $t> 0$ -- 
the $t$-enlargement of $A$ defined by $A_t:=\{x : {\rm dist}(x,A)<t\}$ verifies 
$\mu(A_t) \geq  1-\sqrt{{\pi}/8} \ e^{-n t ^2/2}$.
\end{fact}

The spherical  concentration inequality \eqref{eq:MS} played a huge role in the development of the theory. However, its statement is not completely satisfactory for two reasons.
First, while it is well known that the constant $c=\frac 12$ in the exponent is optimal, stating  Fact \ref{fact:MS} for $S^{n+1}$ sweeps under the carpet the inconvenient truth that the dimension of the ambient space in this formulation is $n+2$, while the factor that appears in the exponent is just $n$. 
It would be more elegant and convenient for that factor to coincide with the dimension of the ambient space,
leading {\em directly} to a concentration result of the form \eqref{fact:MS}.
Next, the constant $C=\sqrt{\pi/8} \approx 0.626657$ in \eqref{eq:MS} is not optimal (ideally, both bounds in Fact \ref{fact:MS} should tend to $\mu(A)=\frac 12$ as $t\to 0^+$). 
Even though it is possible to replace $n+1$ with $n-1$ while adjusting the constants, doing that would only exacerbate the second drawback. 
Here we will prove the following version of the inequality that addresses 
all these concerns.  We emphasize that this bound is not meant to be optimal; in fact, various bounds that are better and near optimal in various asymptotics are known and not that difficult (for example, see Proposition \ref{prop-cap}, its proof, and the comments following it). However, we believe that our results offer a reasonable compromise between sharpness, simplicity, and the ease of application.   

\begin{theorem}\label{sphereconstants}  {\rm (\cite{thesis}) } 
Let $n > 2$ and $t > 0$.
If $A\subset S^{n-1}$ satisfies $\mu(A)\geq 1/2$ then 
\be \label{eq:sphereconstants} 
\mu(A_t )\geq 1- \frac{1}{2}\,e^{-t ^2n/2}.
\ee 
Consequently, if $f: S^{n-1} \rightarrow \R$ is a 
function  which is $L$-Lipschitz with respect to the geodesic distance  and if $M$ is its median, then for every $t \geq  0$,
\be \label{eq:spherelipschitz}
\mu(\left\{f > M+t \right\}) \leq  \frac{1}{2} \, e^{-n t ^2/2L^2} \ 
\hbox{ and } \  \mu(\left\{|f - M| > t\right\}) \leq e^{-n t ^2/2L^2} .
\ee
\end{theorem}

\begin{remark} {\rm Before continuing, let us comment on the case $n=2$, i.e., that of the $1$-dimensional sphere $S^1$.  As easily follows from the general argument sketched in Section~\ref{proof2},  the optimal lower bound in \eqref{eq:sphereconstants} is then very simple and reads (in the nontrivial range $(0,\pi/2]$)
\[
\mu(A_t )\geq\frac{1}{2}+\frac{t }{\pi} . 
\]
A direct check shows then that the estimate \eqref{eq:sphereconstants} from Theorem \ref{sphereconstants} 
fails  if $t  \in (\a,\beta)$, where $\a\approx 1.05858, \beta \approx 1.18588$.
However, it remains true outside of this interval. (This is visualized in Figure \ref{theplot} further below.)  Also, if we use the extrinsic chordal distance in $\R^2$ (or in $\R^n$, for any $n\geq 2$) instead of the geodesic distance, the estimate holds in the entire nontrivial range $ [0,\sqrt{2}]$. Note that the extrinsic distance, i.e., the usual Euclidean distance in the ambient space $\R^n$,   is in many applications more relevant than the geodesic distance. This happens  for example when the function $f$ is defined -- and Lipschitz -- on the entire space  $\R^n$, or at least on the unit ball.} \qed
\end{remark}

Let us now pass to the discussion of bounds of the form \eqref{eq:normal} (or \eqref{eq:spherelipschitz}) with $M=\E f$, the expected value of $f$. The first observation is that, in this context, the value of the constant $C$ can not be smaller than $1$, which is shown by the following simple example. (The example  is cooked up  for $n=2$, but clearly a simple modification with similar features can be produced for any $n\geq 2$, see also Section \ref{sphereconstants2}.)  Identify $S^1$ with $(-\pi,\pi]$ endowed with the {\em normalized} Lebesgue measure, which we will also denote by $\mu$. Next, let $\delta \in (0, \pi)$ and let   $f : [-\pi,\pi] \to \R$ be defined by $f(x)= \min\{|x| - \delta, 0\}$. Then $\E f = -\frac{\delta^2}{2\pi}$ and so, for $t\in (0, \frac{\delta^2}{2\pi})$, 
\[
\mu(\left\{f > \E f+t \right\}) \geq \mu(\left\{f \geq 0\right\}) = 1-\frac{\delta}\pi. 
\]
Accordingly, if -- for such $t$ --  we have $\mu(\left\{f > \E f+t \right\}) \leq C\, e^{-cn t ^2}< C$, then letting $\delta \to 0^+$ yields  $C\geq 1$. 

Thus, when $M=\E f$,  the best bound we may hope for in the estimate \eqref{eq:normal}  is $e^{-n t ^2/2}$.  
Somewhat surprisingly, a stronger fact is true: we also have (for $n>2$) a {\em two-sided } bound of the same form. 

\begin{theorem}\label{sphereconstants2} 
If $f: S^{n-1} \rightarrow \R$ is a function  which is $L$-Lipschitz with respect to the geodesic distance, then, for every $t \geq  0$,
\be \label{eq:spherelipschitz2} 
\mu(\left\{f \geq \E f+t \right\}) \leq  e^{-n t ^2/2L^2} \ \hbox{ for all } \ n\geq 2, 
\ee
\be \label{eq:spherelipschitz3} 
 \mu(\left\{|f - \E f| \geq t  \right\}) \leq e^{-n t ^2/2L^2} \ \hbox{ for all } \ n > 2. 
\ee
\end{theorem} 
\begin{remark} {\rm (i) All the comments presented in the remark following Theorem \ref{sphereconstants} apply {\em mutatis mutandis} to the case $n=2$ of  the inequality from \eqref{eq:spherelipschitz3}.\ (ii) Obviously, the bound 
\eqref{eq:spherelipschitz3} is stronger than \eqref{eq:spherelipschitz2}.  However, we state the latter separately since it is also valid for $n=2$. Moreover, its proof provides some additional information, is a good warmup for the harder proof of \eqref{eq:spherelipschitz3}, and in fact some cases considered in the proof of \eqref{eq:spherelipschitz3} reduce to instances of~\eqref{eq:spherelipschitz2}. 
(iii) An alert reader will recall that a powerful standard tool for obtaining subgaussian estimates for deviations 
from the expected value of a random variable is the log-Sobolev inequality, and will wonder whether at least the one-sided part of the assertion of  Theorem \ref{sphereconstants2}  (i.e., the inequality from \eqref{eq:spherelipschitz2})  can be derived that way. This is almost true, but not quite. 
Indeed, a log-Sobolev inequality for a Riemannian manifold $M$ is usually deduced from a bound on $c(M)$, the Ricci curvature of $M$. Now, $c(S^{n-1}) = n-2$, which by general arguments alluded to above leads to an estimate of the form \eqref{eq:normal} with $M=\E f$,  $C=1$, and the coefficient of $t^2$ in the exponent on the right hand side equal to $\frac{1}2 \frac{(n-1)\,c(S^{n-1})}{n-2}= \frac{n-1}2$.  The stronger two-sided estimate \eqref{eq:spherelipschitz3} presents further problems. }
  \qed
\end{remark} 

A standard consequence of Theorem \ref{sphereconstants2} is the following deviation result for the Gaussian space. 

\begin{corollary}\label{cor:gaussian} 
Let $\gamma=\gamma_n$ be the standard Gaussian measure on $\R^n$ and let $f: \R^n \rightarrow \R$ be an $L$-Lipschitz function (with respect to the Euclidean distance). Next, let $M=\E f$ and $t > 0$. Then 
\be \label{eq:gaussian} 
 \gamma(\left\{|f - M| \geq t  \right\}) \leq e^{-t ^2/2L^2}.
\ee
\end{corollary}

\begin{remark} {\rm (i) When $M$ is the median of $f$, the bound \eqref{eq:gaussian} is an easy consequence of the Gaussian isoperimetric inequality \cite{Borell,ST} and the inequality $\gamma_1\big([t,\infty)\big) \leq \frac 12 e^{-t^2/2}$.  \ (ii) The Gaussian analogue of \eqref{eq:spherelipschitz2} follows from the log-Sobolev inequality via the so-called ``Herbst argument'' (see \cite{Ledoux} or \cite{ABMB}), but we couldn't easily find a reference giving the estimate \eqref{eq:gaussian}  for the two-sided deviation. Indeed, the most frequently cited concentration result of that kind is $\gamma(\left\{|f - \E f| \geq t  \right\}) \leq 2e^{-t ^2/2L^2}$.  While in our presentation Corollary \ref{cor:gaussian} appears as an afterthought, the ubiquity of the normal distribution makes this result potentially at least as useful as Theorems \ref{sphereconstants} and \ref{sphereconstants2}.   }   \qed
\end{remark} 

Our methods are based on refinements of known ideas, a lot of elementary calculus, and some numerics. It would be nice to find better reasons for the results. Hopefully, now that we know that the inequalities hold, someone will come up with a more conceptual proof. Concerning a potential self-contained argument, see the Remark at the end of subsection \ref{sec:one-sided}. Otherwise, there are many relevant techniques that we did not explore seriously, or did not know how to use (for starters: measure transportation and semigroup methods, see \cite{Ledoux}; covariance 
representations, see \cite{Houdre} and \cite{BD}; or the Curvature-Dimension-Diameter condition \cite{EMilman} in the context of Section \ref{products}), and we use only very weakly log-concavity of the marginals of $\mu$ (the measures $\nu=\nu_n$ defined by \eqref{def:nu}), so there is hope.

\section{Deviation from the median: proof of Theorem \ref{sphereconstants}} \label{proof2}
In this section we will sketch the proof of Theorem \ref{sphereconstants}, which, except for the calculus part, follows \cite{MS86}.  
The derivation of the second statement from the first one is standard and well-known (see, e.g., \cite{Ledoux} or \cite{MS86}). 
In fact the two statements are formally equivalent; here is a sketch of the argument.\footnote{Modulo minor modifications, this argument works in any metric probability space.}

First, if $f: S^{n-1} \rightarrow \R$ and $M$ is the median of $f$, 
then the set $A:= \{f\leq M\}$ verifies $\mu(A)\geq \frac{1}{2}$, so  \eqref{eq:sphereconstants} applies. Next, if $f$ is  $1$-Lipschitz, then $\{f > M+t\} \subset S^{n-1}\backslash A_t$, and so any lower bound on $\mu(A_t)$ implies an upper bound on $\mu(\left\{f > M+t \right\})$, which is exactly what we need. 
The second inequality in \eqref{eq:spherelipschitz}  and the case of general Lipschitz constant $L$ follow easily. 

In the opposite direction, if $\mu(A)\geq \frac{1}{2}$, define $\phi(x)= {\rm dist}(x,A)$. Then $\phi$ is $1$-Lipschitz, the median of $\phi$ is $0$, and we have  $\{\phi > t\} = S^{n-1}\backslash A_t$ (for any $t >0$). Again, this means that any upper bound on $\mu(\left\{\phi > t \right\})$ translates to a lower bound on $\mu(A_t)$, as needed.

From this point on we will concentrate on the estimate \eqref{eq:sphereconstants}. 
The spherical isoperimetric inequality guarantees that, 
given $\mu(A)$ and $t>0$,  the value of $\mu(A_t )$ is minimized when $A$ is a  \emph{spherical cap}.
Accordingly, we need only consider the case of a spherical cap $K\subset S^{n-1}$  with $\mu(K) = 1/2$ (i.e., a hemisphere). In other words, we need to show 
 \begin{proposition} \label{prop-main}
If $n>2$ and $x\in [0,\pi/2]$,  then 
\begin{equation} \label{eq:sphereconstants2} 
\mu(S^{n-1}\backslash K_x) \leq \frac{1}{2}e^{-nx^2/2}\ .
\end{equation}
 \end{proposition}

\begin{proof} 
We first note that $S^{n-1}\backslash K_x$ is again a spherical cap (whose radius in the geodesic distance is $r=\pi/2 -x$),  and so --  following \cite{MS86} -- the left hand side of \eqref{eq:sphereconstants2} can be rewritten as

\begin{equation} \label{cap}
\mu(S^{n-1}\backslash K_x) =  (2I_{n-2})^{-1}\int_x^{\pi/2} \cos^{n-2}\theta \, d\theta = (2 I_{n-2})^{-1}\int_0^{r} \sin^{n-2}\theta \, d\theta \ =: v(r),
\end{equation}
where  $I_m$ is the well-known Wallis integral
\be \label{I_m}
I_m:= \int_{0}^{\pi/2} \cos^m\theta \, d\theta .
\ee
[The precise value of $I_m$ is not important for the present argument, but for future reference we will cite 
some easy and well-known facts in Proposition \ref{wallis} at the end of this section.]
This means that the first assertion of  Theorem \ref{sphereconstants} is equivalent to the inequality
\begin{equation} \label{qn} 
q_n(x):= \frac{\int_x^{\pi/2} \cos^{n-2}\theta \, d\theta}{ I_{n-2}} e^{nx^2/2}  \leq 1 ,
\end{equation}
to be valid for $n>2$ and $x\in [0,\frac\pi 2]$. 
While numerical considerations suggest that, for each $x\in [0,\pi/2]$,  the sequence $(q_n(x))$ is nonincreasing,  
 in view of the recurrence formula 
\begin{equation} \label{rec}
\int\cos^n\theta \, d\theta =  \frac 1n \cos^{n-1}\theta \sin \theta + 
 \frac {n-1}n  \int \cos^{n-2}\theta \, d\theta ,
\end{equation}
it will be easier to compare $q_{n+2}(x)$ and $q_{n}(x)$.  Specifically, we will aim at proving that 
\be \label{n+2n}
q_{n+2}(x) \leq q_{n}(x),
\ee
 which 
simplifies to 
\begin{equation}  \label{nn-2}
\int_x^{\pi/2} \cos^n\theta \, d\theta 
\stackrel{?}{\leq}  e^{-x^2}\frac{I_n}{I_{n-2}}  \int_x^{\pi/2}  \cos^{n-2}\theta \, d\theta.
\end{equation}
Passing to the definite integrals   $\int_0^{\pi/2}$  in the  formula \eqref{rec} yields 
$\frac{I_n}{I_{n-2}}  =  \frac {n-1}n$. 
Substituting this value in \eqref{nn-2}, and further applying  the recurrence formula  \eqref{rec}  to the left hand side of \eqref{nn-2},  allows us to rewrite that inequality as an upper bound on  
$\int_x^{\pi/2}  \cos^{n-2}\theta  \, d\theta$, namely as
\begin{equation} \label{upper}
(n-1)(1-e^{-x^2}) \int_x^{\pi/2}  \cos^{n-2}\theta  \, d\theta \stackrel{?}{\leq}  \cos^{n-1}x \sin x \ .
\end{equation}
The cosine integral appearing above can be upper-bounded as follows: 
$$
\int_x^{\pi/2}  \cos^{n-2}\theta   \, d\theta \leq \frac 1 {\sin x} \int_x^{\pi/2}  \cos^{n-2} \theta \sin \theta  \, d\theta = \frac{1}{n-1}\frac{\cos^{n-1} x}{\sin x}\ .
$$
Applying this bound to the left hand side of \eqref{upper}, we see that it is now sufficient to show that 
\begin{equation}\label{ineq}
(1-e^{-x^2}) \stackrel{?}{\leq}  \sin^2 x\ \ \ \hbox{or }\ \ \  \cos x \stackrel{?}{\leq}  e^{-x^2/2}
\end{equation}
for $x\in[0,\pi/2]$, which is a well known inequality used often in analysis. This inequality can be validated in many ways; for example, one may consider the power series of both sides, or take logarithm of both sides and repeatedly differentiate. 

To recapitulate, we have shown up to now that, for each $x\in [0,\pi/2]$,  the sequences  $(q_{2n}(x))$ and  
$(q_{2n+1}(x))$ are nonincreasing.
Thus, to deduce \eqref{qn},  we only need to establish that, for  $x\in[0,\pi/2]$, $q_3(x)\leq 1$ and $q_4(x)\leq 1$. 
(The failure of $q_2(x)\leq 1$ on some subinterval  $(\a,\beta)\subset [0,\pi/2]$ was the reason 
why the case $n=2$ had to be excluded from Theorem \ref{sphereconstants} 
and analyzed separately.) 
From \eqref{qn}, $q_3(x)$ and $q_4(x)$ are given by the following
\begin{align}
q_3(x) &= (1-\sin x) e^{3x^2/2}, \label{q3}\\
q_4(x) &= \frac{1}{\pi}\left(\pi - 2x - 2\cos x \sin x \right)e^{2x^2} .
\end{align}
The inequalities $q_3(x)\leq 1$ and $q_4(x)\leq 1$ 
can now be verified numerically or graphically, see Figure \ref{theplot}. 
Note that the only points where $q_3(x)$ or $q_4(x)$ is even close to $1$ are near $x=0$ 
(for which we have equality), but since $q_3'(0)=-1$ and $q_4'(0)=-\frac{4}{\pi}$, we can be 
sure that the inequalities hold when $x$ is close to $0$. 
(Note that, additionally,  we know that $q_4(x) \leq q_2(x)$, so that $q_4(x)\leq 1$ 
for sure holds outside the short interval $(\a,\beta)$ identified earlier.) 
\begin{figure}[ht!]
\includegraphics[width=0.8\textwidth]{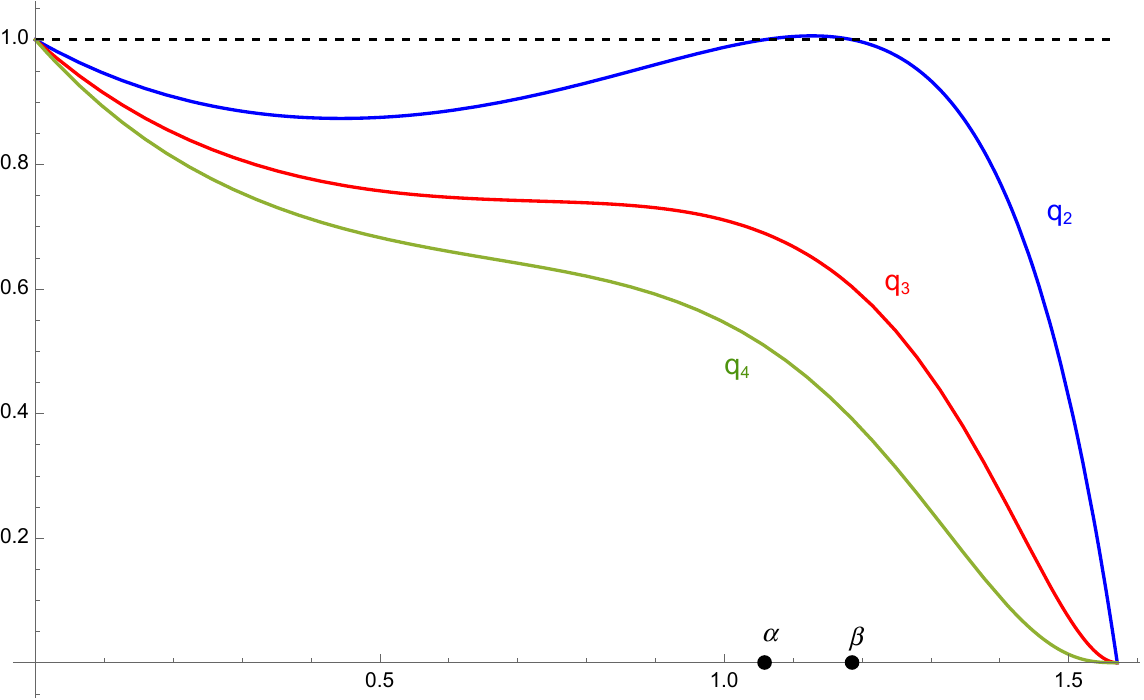}
\caption{The plots of $q_2$, $q_3$, and $q_4$. The bound $q_2\leq 1$ is not valid on some interval $(\a,\beta )\approx (1.05858, 1.18588)$, but $q_4,  q_3\leq 1$ (analytically) and  $q_4\leq q_3$, $q_3\leq q_2$ (numerically).} 
\label{theplot}
\end{figure}

Alternatively, we can analyze 
the inequalities by analytic means using the same techniques as those employed earlier 
in the proof of \eqref{ineq}.  
However, the argument is more involved and at some stage numerics seem necessary. 
For example, 
taking the logarithm of both sides of $q_3(x)\leq 1$ leads to 
\[
\ln(1-\sin x)\leq\frac{-3x^2}{2} .
\]
We again have equality when $x=0$ so we look at the derivatives and are led to 
\[
g(x):= 3x - \frac{\cos x}{1-\sin x}\leq 0
\]
(that is, if the above inequality holds in $(0,\pi/2)$, then so does the previous one and we are done). 
By direct calculation, $g'(x) = \frac{2-3\sin x}{1-\sin x}$, and so it is apparent 
that $g$ has a unique maximum in $[0,\pi/2)$ at $x_0=\arcsin \frac 23$. 
It remains to verify that $g(x_0) \approx -0.046885 < 0$, as needed.  \end{proof}

\begin{remark} {\rm  It is possible to generalize the above argument to arrive at estimates involving  $\frac{1}{2}e^{-t ^2(n+\xi)/2}$, where $\xi>0$. 
We cannot expect such estimates to hold for all $n$ :  we have already seen above that, even in the case $\xi=0$, $n$ must be greater than $2$. 
However, they may be true for $n\geq n(\xi)$. Bounds of this type are relevant to improving constants in isoperimetric/concentration inequalities on $\big(S^{n-1}\big)^k$ for $k >1$, the topic that is explored in Section \ref{products}. } \qed
 \end{remark} 
 
 For future reference, we will present here other closely related bounds for the quantities appearing in Theorem \ref{sphereconstants}.  For simplicity, we will state them only as an upper bound for the volume of spherical caps; estimates of the form \eqref{eq:sphereconstants}  and \eqref{eq:spherelipschitz} follow then in the usual way. 
 
 \begin{proposition} \label{prop-cap}
For $0\leq r \leq \pi/2$, the normalized volume of a spherical cap of geodesic radius $r$ in $S^{n-1}$ 
satisfies 
\be \label{cap-cos}
v(r) \leq  \frac 12  \sin^{n-1} r 
\ee
and, for $0\leq r < \pi/2$, 
\be \label{cap-cos2}
v(r) \leq  (\sqrt{2\pi}\, \kappa_n \cos r)^{-1} \sin^{n-1} r, 
\ee
where $\kappa_n = \frac{\sqrt{2 } \; \Gamma \left(\frac{n+1}{2}\right)}{\Gamma \left(\frac{n}{2}\right)}$. Moreover, the ratio $v(r)/ \sin^{n-1} r $ is increasing on $(0,\pi/2]$.
 \end{proposition}
 \begin{remark} {\rm The inequality \eqref{cap-cos} is equivalent to $\mu(S^{n-1}\backslash K_x) \leq  \frac 12  \cos^{n-1} x$, where $x= \pi/2 -r$, which is superior to \eqref{eq:sphereconstants2} except for $x$ close to $0$. Next, $n-1= \dim S^{n-1}$ is the best exponent that we may possibly expect since, for small $r$, the volume of  an $r$-cap scales as $r^{n-1}$.  Finally, since $\kappa_n > \sqrt{n - \frac 12}$, the bound \eqref{cap-cos2} is superior to \eqref{cap-cos} except for $r$ very close to $\pi/2$; this will be exploited in the proof of inequality \eqref{eq:spherelipschitz3} from Theorem \ref{sphereconstants2}. } \qed
 \end{remark} 
 \begin{proof}  First, it is apparent that we have equality in \eqref{cap-cos} when $r=0$ and $r=\pi/2$. 
 Accordingly, \eqref{cap-cos} will follow once we prove the last statement, i.e., $r \to v(r)/ \sin^{n-1} r $ being increasing on $[0,\pi/2]$.  To that end, recall that (cf.\ \eqref{cap}) 
 $$
 v(r) = (2I_{n-2})^{-1}\int_0^{r} \sin^{n-2}\theta \, d\theta , 
 $$
 while $ \sin^{n-1} r  = (n-1) \int_0^{r} \sin^{n-2}\theta \cos \theta\, d\theta$.  The conclusion follows now immediately  from the following elementary fact.
 \begin{lemma} Let $f : (a,b) \to (0,\infty)$ be integrable and let $h : [a,b) \to \R$ be nonincreasing. Then 
 the function $x \to \frac{\int_a^x fh}{\int_a^x f}$ is nonincreasing. 
 \end{lemma}
\noindent Let us note that the Lemma was stated here with the hypotheses fitting the current setting, but it remains true under any reasonable assumptions that assure the quantities in question are well defined. 

We skip the proof of \eqref{cap-cos2}; it can be found, together with  
a geometric proof of \eqref{cap-cos} and a few similar or related estimates,  in Section 5.1.2 and Appendix A  of \cite{ABMB}. Similar bounds were also established, e.g.,  in  Lemma 2.1 of \cite{BGKKLS}. 
 \end{proof} 

 Finally, let us state for future reference some easy and well-known facts concerning the Wallis integral defined in \eqref{I_m}. 
 
  \begin{proposition} \label{wallis}
If  $I_m = \int_{0}^{\pi/2} \cos^m\theta \, d\theta$ then  \\
{\rm (i)} $I_m = \frac{\sqrt{\pi }\, \Gamma \left(\frac{m+1}{2}\right)}{2 \Gamma \left(\frac{m}{2}+1\right)}$\  for \ $m\geq 0$, \\
{\rm (ii)} $\sqrt{\frac{\pi}{2m+2}}\leq I_m \leq \sqrt{\frac{\pi}{2m+1}}$  \ for \ $m\geq 1$.\\
{\rm (iii)}  We have $I_m= \sqrt{\frac{\pi}2}\, \kappa_{m+1}^{-1}$, where $\kappa_s$ is as in Proposition \ref{prop-cap}.  The sequence $\left(\frac{\kappa_s}{\sqrt{s}}\right)$ is increasing to $1$, and so  $\big(I_{m-1} \sqrt{m}\big)$ and $\big(I_{m-2} \sqrt{m}\big)$ both decrease to  $\sqrt{\frac\pi 2}$. 
  \end{proposition}
\noindent  For these (and other, tighter and two-sided) estimates we refer the reader to, for example, Section 5.1.2 and Appendix A in \cite{ABMB}.

\section{Deviation from the mean: proofs of Theorem \ref{sphereconstants2} and Corollary \ref{cor:gaussian}}
In this section we will address the one-sided and the two-sided problem for the deviation from the mean:  {\em If $f$ is $1$-Lipschitz and $t\geq0$, what are the bounds for 
$ \mu(\{f \geq \E f +t\}) $ and for $ \mu(\{|f - \E f | \geq t\}) $?}  The bulk of the argument will be devoted to the spherical case (Theorem  \ref{sphereconstants2}); at the very end we will make a few comments about deducing the Gaussian result (Corollary \ref{cor:gaussian}). 

As  in the previous section, the proof of Theorem \ref{sphereconstants2} splits into two parts:  
(i) identifying extremal instances for the problem at hand and 
(ii) obtaining tight estimates satisfied by those extremal instances.   
In the case  of Theorem \ref{sphereconstants}, the extremal objects were the {\em spherical cap} $K$ of measure $\frac 12$, i.e., a hemisphere 
(in the context of \eqref{eq:sphereconstants}) and the function $\phi(x)= {\rm dist}(x,K)$ (in the context of  \eqref{eq:spherelipschitz}).  
This will not suffice in the present setting since $\E f$ depends on the entire distribution of $f$, i.e., on the values of $ \mu(\{f \geq t\})$ 
for all $t\in \R$, and not only on the values of $t$ for which that measure is $\frac 12$. 
However, allowing caps of arbitrary measure leads to a sufficiently rich family of functions, which we will describe in a moment. 

\subsection{The one-sided problem : Proof of \eqref{eq:spherelipschitz2}} \label{sec:one-sided} \ 
We will start with the following simple lemma. 
\begin{lemma} \label{extreme} 
Let $\nu$ be a probability measure on $\R$ such that $\int |x| d\nu(x) < \infty$. 
Next, let $\mathcal{L}$ be the set of functions from $\R$ to $\R$ that are $1$-Lipschitz and nondecreasing. 
For $t > 0$,  consider the optimization problem
\begin{equation} \label{opt}
\sup_{\psi \in \mathcal{L}}\ \nu\big(\{\psi \geq \E \psi +t\}\big),  
\end{equation}
where $\E$ stands for the expected value in the probability space $(\R, \nu)$.  Then, for each $t$, it is enough to restrict the supremum to the subfamily of $\mathcal{L}$ consisting of functions of the form 
\be \label{phi_a}
\phi_a(\theta) = \left\{ \begin{array}{cl}\theta&{\rm if } \quad \theta \leq a\\
a&{\rm if } \quad \theta \geq a 
\end{array} \right. ,
\ee
where  $a \in \R$ is a parameter.  Moreover, for each $a$ it is sufficient to consider $t= t(a):= a-\E \phi_a$. 
That is, if $\lambda : (0, \infty)\to [0,1]$ is a nonincreasing function satisfying $\nu(\{\phi_a \geq \E \phi_a + t(a)\}) \leq \lambda(t(a))$ for all $a\in \R$, then it is also true that $\nu(\{\psi \geq \E \psi + t\}) \leq \lambda(t)$ for all $t>0$  and all $\psi\in \mathcal{L}$.
\end{lemma}
\begin{proof} 
Denote $U(\psi,t) := \{\psi \geq \E \psi +t\}$. For the first assertion of the Lemma, we need to show that, for $t > 0$, 
\begin{equation} \label{opt2}
\sup_{\psi \in \mathcal{L}}\ \nu\left(U(\psi,t) \right) \leq  \sup_{a \in \R}\ \nu\left(U(\phi_a,t) \right), 
\end{equation}
the converse inequality being trivial.   
 To that end, fix $\psi \in \mathcal{L}$ and $t > 0$.  
 Since $\psi$ is nondecreasing an continuous, the set $U(\psi,t)$ is either empty, or of the form $[a_0,\infty)$ for some $a_0\in \R$,  and necessarily $f(a_0)=\E \psi +t$. 
 Next, since adding a constant to any function $\zeta$ doesn't change the set $U(\zeta,t)$, we may just as well assume that $\psi(a_0)=a_0$. 

\smallskip Consider now the function $\phi_{a_0}$ as defined by \eqref{phi_a}. Then 

\smallskip\noindent
$\bullet$ $\psi(x) \geq \phi_{a_0}(x) = a_0$  for $x\geq a_0$ with equality for $x=a_0$  (by  construction)\\
$\bullet$ $\psi(x) \geq \phi_{a_0}(x) = x$  for $x \leq   a_0$ (because $\psi$ is $1$-Lipschitz).  \\
(The reader is advised to draw a picture.) 
Thus $\psi \geq \phi_{a_0}$ everywhere on  $\R$ and it follows in particular that $\E\phi_{a_0} \leq \E \psi$. 
Consequently, if we choose $t_0$ so that $\E\phi_{a_0} +t_0=\E \psi +t$, then $t_0\geq t$. 
On the other hand, $U(\phi_{a_0},t_0) = \{\phi_{a_0} \geq a_0\}=[a_0,\infty)=U(\psi,t)$, and since for any $g$ and any $t\leq t_0$ 
we obviously have $U(g,t_0) \subset U(g,t)$, it follows that 
\[
\nu(U(\phi_{a_0},t)) \geq \nu(U(\phi_{a_0},t_0)) = \nu(U(\psi,t)). 
\]
Since $\psi \in \mathcal{L}$ was arbitrary, \eqref{opt2} follows. 
The second assertion of the Lemma is a consequence of the fact that to upper-bound $\nu(U(\psi,t))$ we only need information about 
$U(\phi_{a_0},t_0)$ for some $t_0\geq t$, wiith the equality $t_0=t(a_0)$ being implicit in the definition of the latter set. 
\end{proof} 
\begin{remark} {\rm  The second assertion of Lemma \ref{extreme} can be strengthened and simplified as follows}:  For any $t>0$, the supremum from \eqref{opt} is {\em attained} for some $\psi=\phi_a$ with $a$ verifying $t= a-\E \phi_a$. {\rm We stated the weaker version since it is sufficient for our purposes and easier to prove, but since the Lemma may be of independent interest, we include a sketch of the proof of the stronger fact in Appendix \ref{achieve}.}   
\qed
 \end{remark} 
 We will now sketch a very well known reduction argument that allows to derive from Lemma \ref{extreme} the form of the extremal functions 
 for the one-sided problem \eqref{eq:spherelipschitz2} from Theorem \ref{sphereconstants2}.  (The same argument will work for the two-sided problem \eqref{eq:spherelipschitz3} once 
 we establish a two-sided analogue of Lemma \ref{extreme}.) 

Let $f$ be any (say, Borel) function on $S^{n-1}$ and let $f^*$ be its rearrangement (i.e., verifying $\mu(\{f^* \geq t\}) = \mu(\{f \geq t\})$ for any $t\in \R$) 
that is of the form
\be \label{rearrange}
f^*(u) = g(u_1),
\ee
where $u_1 =\ip{u}{e_1}$ is the first coordinate of $u \in S^{n-1}$ and $g: [-1,1] \to \R$ is nondecreasing. 
This is a standard procedure that works for any random variable and any non-atomic probability measure on $\R$,  
but in our setting it has an additional feature: if $f$ is $1$-Lipschitz, so is $f^*$.  Indeed, suppose $f$ is $1$-Lipschitz 
and let $v, w \in S^{n-1}$. We need to show that  if $t=f^*(w)=  g(w_1)$, $s=f^*(v)=g(v_1)$, then $\ep:= |t-s| \leq {\rm dist}(v,w)$. 
 
By symmetry, we may assume that $s<t$ (hence $v_1< w_1$). By construction, the sets $K=\{u\in S^{n-1} : f^*(u) \leq s\}$ and 
 $L=\{u\in S^{n-1} : f^*(u) \geq t\}$  are ``opposite'' spherical caps with ``parallel'' boundaries and with $v\in K, w\in L$. 
 Likewise, by construction, if $A =\{f\leq s\}$ and $B =\{f\geq t\}$, then  $\mu(K) = \mu(A)$ and $\mu(L) = \mu(B)$. 
 Next, since $f$ is $1$-Lipschitz, it follows that ${\rm dist}(A,B) \geq \ep$; in other words, $B \subset (A_\ep)^c$, 
 where $A_\ep$ is the $\ep$-enlargement of $A$ defined in Fact \ref{fact:MS}. We now appeal to the isoperimetric inequality 
 on $S^{n-1}$ to conclude that $\mu\big((K_\ep)^c\big) \geq \mu\big((A_\ep)^c\big) \geq \mu(B) = \mu(L) $.  
 Since both $L$ and  $(K_\ep)^c$ are (closed) ``left'' spherical caps, we deduce  that $L \subset (K_\ep)^c$ or, 
equivalently,  ${\rm dist}(K,L) \geq \ep$, and -- in particular -- ${\rm dist}(v,w) \geq \ep$, as needed.
 
Since all quantities depending on the distribution are identical for $f$ and $f^*$, it follows that for estimates 
such as \eqref{eq:spherelipschitz2}  it is enough to consider functions of the form \eqref{rearrange}.  
Further, since the geodesic distance between the two ``parallels'' $\{u \in S^{n-1} : u_1= \alpha\}$ and $\{u \in S^{n-1} : u_1= \beta\}$
equals $|\arcsin \alpha - \arcsin \beta|$, the function $f^*$ defined by \eqref{rearrange} is $1$-Lipschitz on $S^{n-1}$ if an only if 
$\psi (\theta) := g(\sin\theta)$  is $1$-Lipschitz on $[-\pi/2,\pi/2]$.  Putting all these observations together, we  conclude that 
the one-sided bound   \eqref{eq:spherelipschitz2}  for given $n\geq 2$ is an instance of \eqref{opt} with 
$\nu$ being the push-forward of $\mu=\mu_n$ under the map  $S^{n-1} \ni u \to \arcsin u_1 \in [-\pi/2,\pi/2]$, 
with the extremal functions given by \eqref{phi_a}.\footnote{Note that the functions in \eqref{phi_a}  
are a priori defined on $\bbR$, but only their restrictions to  $[-\pi/2,\pi/2]$ and the values $a \in [-\pi/2,\pi/2]$ are relevant to the problem at hand. 
However, for other reference measures (for example, the Gaussian measure), 
test functions defined on $\bbR$ and all values of $a\in \R$ may be needed.} 
 
\smallskip 
As was (implicitly) determined in Section \ref{proof2}, $\nu=\nu_n$ is then of the form 
\be \label{def:nu}
 d\nu(\theta) = (2I_{n-2})^{-1}\cos^{n-2}\theta \, d\theta ,\quad \theta \in [-\pi/2,\pi/2] ,
\ee
where $I_m$ is defined by \eqref{I_m}.\footnote{For definiteness, we will assume the that the density of $\nu$ is $0$ outside of the interval $[-\pi/2,\pi/2]$.} Accordingly, verifying the one-sided estimates from Theorem \ref{sphereconstants2} numerically for ``not-too-large''  $n$,
and analytically for  ``small''  $n$, is completely straightforward. 
First, given $a \in [-\pi/2, \pi/2]$ and $t\geq 0$ such that $\E \phi_a + t = a$, the set $\{\phi_a \geq \E \phi_a + t\}$ is exactly $[a, \infty)$, 
and its measure is 
$$
\nu\big([a, \pi/2]\big) = (2I_{n-2})^{-1}\int_a^{\pi/2} \cos^{n-2}\theta \, d\theta,  
$$ 
which is the Haar measure of the corresponding cap 
\be 
K^a = \{u \in S^{n-1} : u_1 \geq \sin a\},
\ee  
already analyzed in Section \ref{proof2} (at least for $a\geq 0$, see Proposition \ref{prop-main}; note that, in the notation from that section, $K^a=S^{n-1}\setminus K_a$). 
This quantity (as a function of $a\in [-\pi/2,\pi/2]$) needs to be compared with 
$e^{-nt^2/2}$, where  $t= t_n(a) = a -\E \phi_a$, which 
can be rewritten as  
\be \label{eq:t}
 t=  a + \E (\theta-a)^+= a+ (2I_{n-2})^{-1}\int_{\theta=a}^{\pi/2} (\theta -a) \cos^{n-2}\theta \, d\theta ,
\ee
where $\theta$, $(\theta-a)^+$ are understood as random variables in the probability space $(\R, \nu)$, and we use the fact that $\E \theta = 0$. 
  For future reference, let us note that  
\eqref{eq:t} can be restated as follows
\be\label{eq:tAlt}
 t= 
 a+ \int_a^{\pi/2} \nu(\{\theta > x\}) \, dx = a+ \int_a^{\pi/2} \mu(K^{x}) \, dx \, ;
\ee
this is because for any nonnegative random variable $X$ one has $\E X = \int_0^\infty \P(X>x)\, dx$. 

Note that if $a$ is close to (but strictly greater than) $-\pi/2$, then $\phi_a \equiv a$ on a set of nearly full measure, while $\E \phi_a < a$. Consequently, for $t=a-\E \phi_a > 0$ we have  $ \mu(\{\phi_a \geq \E \phi_a +t\}) = \mu(\{\phi_a \geq a\}) \approx 1$, while 
$e^{-nt^2/2} < 1$ (in fact also necessarily $ \approx 1$).  This is another argument showing that, in the present context, one can not hope for the multiplicative constant $C$ in the bound of type \eqref{eq:normal} to be strictly smaller than $1$.   

To summarize, we have shown that the validity of the one-sided bound from \eqref{eq:spherelipschitz2} for given $n$ will follow from (and in fact is equivalent to) 
\be\label{eq:one-sided}
\mu(K^a) \stackrel{?}{\leq} e^{-nt^2/2}, 
\ee
where $t= t_n(a)$  is defined via \eqref{eq:t} or  \eqref{eq:tAlt}. Here is a sketch of the calculation showing 
that \eqref{eq:one-sided} holds for $a\geq 0$. 
(In fact, we will see that in that range a tighter bound, with a better constant $C<1$, can be found.  
The argument from the proof of Lemma \ref{extreme} implies then that a version of \eqref{eq:spherelipschitz2} 
with that improved constant $C$ holds for all $1$-Lipschitz functions and all $t$ above the threshold 
given by \eqref{eq:t} or  \eqref{eq:tAlt} with $a=0$. The so calculated threshold depends on $n$ and is asymptotically 
equivalent to $(2\pi n)^{-1/2}$.)

Since we know from 
Proposition \ref{prop-main}  (at least for $n>2$, which we assume) that the measure of the cap 
in question does not exceed $\frac{1}{2}e^{-na^2/2}$, it is enough to show that 
\be \label{a>0}
 \frac{1}{2}e^{-na^2/2} \stackrel{?}{\leq} e^{-nt^2/2} = e^{-n(a+\eta)^2/2} ,
 \ee
where 
\be \label{delta}
 \eta := t-a = -\E \phi_a=\int_a^{\pi/2} \mu(K^{x}) \, dx
 \ee
(cf.\  \eqref{eq:tAlt}).  
Taking logarithms of both sides of the inequality \eqref{a>0} shows that it is equivalent to 
\[
(a+\eta)^2 \stackrel{?}{\leq} a^2 + \frac{2 \ln 2}n 
\]
and finally to 
\[
\eta  \stackrel{?}{\leq}  \sqrt{a^2+ \frac{2 \ln 2}n } -a .
\]
\otoh appealing again to the estimate $ \mu(K^{x}) \leq \frac{1}{2}e^{-nx^2/2}$,   
we can upper-bound $\eta$  by $\frac{1}{2} \int_a^{\pi/2} e^{-nx^2/2} \, dx < \frac{1}{2} \int_a^{\infty} e^{-nx^2/2} \, dx$, so it is enough to show that, for any $a\geq 0$ and $n > 2$, 
\[
\frac{1}{2} \int_a^{\infty} e^{-nx^2/2} \, dx  \stackrel{?}{\leq}  \sqrt{a^2+ \frac{2 \ln 2}n } -a . 
\] 
Let us now change variables via $a = \frac u{\sqrt{n}}$ and, inside the integral, 
$x=\frac s{\sqrt{n}}$ to get an equivalent dimension-free form
\[
\frac{1}{2} \int_u^{\infty} e^{-s^2/2} \, ds  \stackrel{?}{\leq}  \sqrt{u^2+ {2 \ln 2}} -u . 
\]
This inequality is easy to confirm, in fact a much sharper bound 
$\frac{1}{2} \int_u^{\infty} e^{-s^2/2}  < (\sqrt{u^2+ 1} -u) e^{-u^2/2}$ follows from
the well-known Komatu inequality (\cite{Komatu}, or see Remark 4 in \cite{corr})
\be \label{komatsu}
\int_u^{\infty} e^{-s^2/2}  ds \leq \frac {2  e^{-u^2/2}}{u+\sqrt{u^2+2}} .
\ee 
As is easy to check, the above argument  yields (for $a\geq 0$ and $n>2$)   the bound in \eqref{eq:one-sided}
that is of the form $C e^{-nt^2/2}$ with $C=\frac{e^{1/2}}2\approx 0.8244 < 1$.  
Further improvement is possible if one 
replaces the use of the Komatu inequality \eqref{komatsu} by the more precise bound 
$\int_u^{\infty} e^{-s^2/2}   \leq \frac {4  e^{-u^2/2}}{3u+\sqrt{u^2+8}}$ (\cite{Sampford}, or see Proposition 3  in \cite{corr} ; 
numerical check suggests that the constant  $C=0.53$ 
works. This shows that -- except for small values of $t>0$ -- the bound   
\eqref{eq:spherelipschitz2} is not very sharp. However, this is a feature, not a bug:   
it provides the wiggle room needed to deliver the two-sided bound \eqref{eq:spherelipschitz2}. 

Concerning the case $n=2$ of  \eqref{eq:one-sided}, the verification is -- as pointed out earlier -- completely straightforward. 
Indeed, a direct computation leads to $\mu(K^a)=\frac 12 -\frac a\pi$ and $t=t(a)=\frac{(a+\pi/2)^2}{2\pi}$ 
and there is no doubt that \eqref{eq:one-sided} holds in the entire non-trivial range $-\pi/2 \leq a \leq \pi/2$, 
with equality iff $a=-\pi/2$. 


\medskip   For $a<0$, we (trivially) have equality in \eqref{eq:one-sided} when $a = -\pi/2$ (and $t=0$, for all $n$), 
but otherwise the bound $e^{-nt^2/2}$ does not seem very tight. 
(This can be seen heuristically by expanding the quantities in question in powers of $z=a+\pi/2$ if $z$ is small, and approximating the random variable $\sqrt{n}\, \theta$ by a standard normal random variable if $\theta$ is not ``too large.'') 
For a rigorous argument, observe first  that, for $a<0$, 
it is more transparent to rewrite the formula for $t$ as
\be\label{eq:tAlt2}
 t= a+ \E (\theta-a)^+ = \E (\theta-a)^- =  \int_{-a}^{\pi/2} \nu(\{\theta > x\}) \, dx =  \int_{-a}^{\pi/2} \mu(K^x) \, dx .
\ee
In other words, we need to show that if $b:=-a \in [0,\pi/2]$, then
\be\label{eq:tAlt3pre}
t=  \int_{b}^{\pi/2} \mu(K^{x}) \, dx \quad  \stackrel{?}{\Rightarrow} \quad 1- \mu(K^{b}) \leq e^{-nt^2/2} ,
\ee
where we used $\nu\big([a,\infty) \big) = 1-\nu\big([b,\infty) \big)=1- \mu(K^{b})$. 
Note that, in the present context, the relationship between $t$ and $b$ is exactly the same as the relationship between $\eta$ and $a$ was in the case $a\geq 0$.  Next observe that 
in \eqref{eq:tAlt3} we are in a different regime than for $a \geq 0$~:  unless $b>0$ is very small, both sides in the last inequality are close to $1$. Accordingly, it is more appropriate to restate the conclusion of \eqref{eq:tAlt3pre} as a lower bound on the probability of the complementary event
\be\label{eq:tAlt3}
t=  \int_{b}^{\pi/2} \mu(K^{x}) \, dx \quad  \stackrel{?}{\Rightarrow} \quad \mu(K^{b}) \geq 1- e^{-nt^2/2} ,
\ee
To further facilitate concentrating on the values of $b$ that are close to $\pi/2$, we change variables to $y:=\pi/2-x$ and $\alpha =\pi/2-b$.  
The statement \eqref{eq:tAlt3} becomes then
\be\label{eq:tAlt4}
t=  \int_{0}^{\alpha} v(y)\, dy \quad  \stackrel{?}{\Rightarrow} \quad v(\alpha) \geq 1-e^{-nt^2/2} ,
\ee
where \be\label{eq:tAlt5}
v(r)=v_n(r):= \mu(K^{\pi/2-r}) =  (2I_{n-2})^{-1} \int_0^r \sin^{n-2} \theta\, d\theta
\ee
 is the normalized volume of a spherical cap of geodesic radius $r$ in $S^{n-1}$, the function that was already defined in \eqref{cap}. We now appeal to Proposition \ref{prop-cap} to obtain 
\be \label{eq:tAlt6}
 t=t_n(\alpha) \leq \frac 12  \int_{0}^{\alpha} \sin^{n-1} y\, dy = I_{n-1} v_{n+1} (\alpha). 
\ee
The next step is the following simple observation.
\begin{lemma} \label{v_k-v_k+1}
If $\alpha \in [0,\pi/2]$ and $k\geq 2$, then $v_{k+1}(\alpha) \leq v_k(\alpha)$.
\end{lemma}

\begin{proof}
There is equality for $\alpha=0$ and $\alpha=\pi/2$, and 
the function $x \to v_k(x)-v_{k+1}(x)$ has the derivative $\frac{\sin^{k-2} x}{2I_{k-2}}-\frac{\sin^{k-1} x}{2I_{k-1}}$, which is positive on $[0,\alpha_0]$ and
negative on $[\alpha_0,\pi/2]$, where $\alpha_0$ is such that $\sin \alpha_0 = I_{k-1}/I_{k-2}$.
\end{proof}
Combining Lemma \ref{v_k-v_k+1} and \eqref{eq:tAlt6} we see that \eqref{eq:tAlt4} will follow if 
\be
v_n(\alpha) \stackrel{?}{\geq} 1-e^{-n(I_{n-1} v_{n} (\alpha)^2/2} .
\ee 
Since $e^{-u} \geq 1-u$, the above can be further strengthened to
\be
v_n(\alpha) \stackrel{?}{\geq} n\big(I_{n-1} v_{n} (\alpha)\big)^2/2, 
\ee 
which in turn is equivalent to
\be
nI_{n-1}^2 v_{n} (\alpha)  \stackrel{?}{\leq} 2. 
\ee 
This is evidently true for $n\geq 2$ since  $v_{n} (\alpha) \leq v_{n} (\pi/2) = \frac 12$ and $I_{n-1}^2 \leq \frac{\pi}{2(n-1)}$ by Proposition \ref{wallis}(ii), which concludes the proof of \eqref{eq:spherelipschitz2}. 
\begin{remark} {\rm We point out that the two crucial inequalities appearing in the proof, namely \eqref{eq:one-sided} and the conclusion of \eqref{eq:tAlt3}, are {\em de facto} functional inequalities relating the function $\eta(\cdot)$ and its derivative or, in the formulation in the spirit of \eqref{eq:tAlt4}, the function $v_n(\cdot)$ and its primitive.  
Accordingly, it is conceivable that once one comes up with a manageable related differential inequality, these functional inequalities would follow. (This could be parallel to the proofs of Komatu-like inequalities, see, e.g., Proposition 3 in \cite{corr}, or Exercise A.2 in \cite{ABMB}.)  Similar comments apply to the proofs of the two sided-bound \eqref{eq:spherelipschitz3} and Corollary \ref{cor:gaussian} in the next two subsections. \qed
 }   
\end{remark}

\bigskip 
\subsection{The two-sided problem : Proof of \eqref{eq:spherelipschitz3}} \  We now pass to the analysis of the estimate \eqref{eq:spherelipschitz3} from Theorem \ref{sphereconstants2}, i.e., the bound for 
$\mu(\{|f - \E f | \geq t\})$. The initial step, a reduction to the case of $\psi \in \mathcal{L} = \{\psi : \R \to \R, \psi \hbox{ is }\ 1\hbox{-Lipschitz and nondecreasing}\}$ 
and the reference measure $\nu=\nu_n$ defined by \eqref{def:nu} is the same as for the one-sided problem. 
The second step, 
the analogue of Lemma \ref{extreme}, i.e.,  a reduction to functions $\phi_a$ from \eqref{phi_a}, is slightly more involved, and  we impose some mild restrictions on   
 the reference measure $\nu$, which needs to be symmetric and unimodal 
(by the latter we mean that $d\nu(x) = \rho(|x|) \,dx$, where $\rho : \R^+ \to \R^+$ is nonincreasing).  

\begin{lemma} \label{extreme2} 
Let $\nu$ be a symmetric unimodal probability measure on $\R$ such that $\int |x| d\nu(x) < \infty$ and let $t>0$. Then 
\begin{equation} \label{opt3}
\sup_{\psi \in \mathcal{L}}\ \nu\left(\{|\psi - \E \psi| \geq t\}\right) =  \sup_{a\in \R}\ \nu\left(\{ |\phi_a- \E \phi_a| \geq t\}\right). 
\end{equation}
Moreover, for each $a$ it is sufficient to consider $t=t(a) = a-\E \phi_a$.  
\end{lemma}
\noindent The proof of the Lemma is elementary, but on the complicated side. We relegate it to Appendix \ref{lemma11}.

Similarly as was the case in the one-sided setting, Lemma \ref{extreme2}  reduces -- for a specific density $\nu$ -- the inequality of type  \eqref{eq:spherelipschitz3} to a comparison of two concrete functions of the parameter $a$, which can be verified numerically. For added rigor, this should be accompanied by an asymptotic analysis of the quantities in question when $t\to 0$ (note that as $t \to 0$, both sides of  \eqref{eq:spherelipschitz3}  typically converge to $1$) and -- if  $\nu$ is not compactly supported~--~as $a \to \infty$.  In particular, it is routine to check whether  \eqref{eq:spherelipschitz3}  holds for any particular value of $n$. 
For $n=2$ there is a failure on the same interval  for which Theorem \ref{sphereconstants}  failed for $n=2$, and  the failure happens for the same reason:  in the reformulation in terms of $\nu=\nu_2$ given by \eqref{def:nu}, consider the function $\psi(\theta)=\theta$ and note that since the density of $\nu_n$ is constant for $n=2$, replacing the median by the mean does not make any difference. For $n=3$, the integrals involved in the definitions of $\E \phi_a$, $t=t(a)$, and the relevant probabilities can be explicitly evaluated and the resulting graphs look as in Figure~\ref{n=3}. 
\begin{figure}[ht!]
\includegraphics[width=0.8\textwidth]{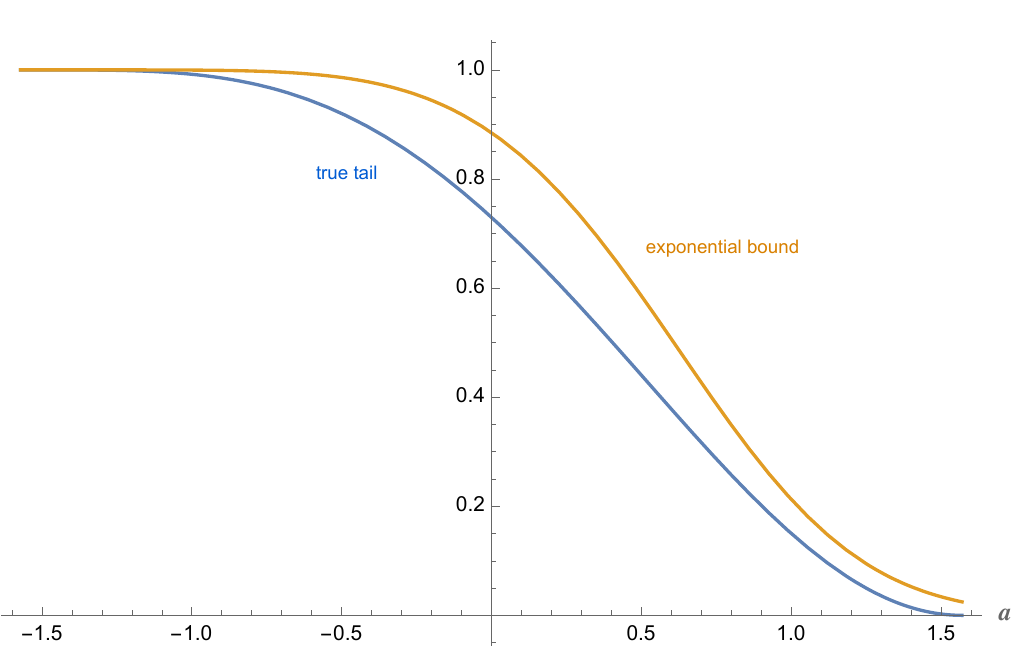}
\caption{The graphs of $a \to F(a) =\nu_n\left(\{ |\phi_a- \E \phi_a| \geq t(a)\}\right)$ vs.  $a\to G(a)=e^{-nt(a)^2/2}$ for $n=3$. There is a clear separation except when $a$ is close to $-\pi/2$.}
\label{n=3}
\end{figure}
\noindent Clearly, the only values of $a$ that are questionable are those close to $-\pi/2$, but it is readily verified that  $F(a)=\nu_3\left(\{ |\phi_a- \E \phi_a| \geq t\}\right) = 1-\frac1{12} (a+\pi/2)^4 + O((a+\pi/2)^6)$, while $G(a)= e^{-nt(a)^2/2} = 1-\frac1{96} (a+\pi/2)^6+O((a+\pi/2)^8)$. Consequently, comparing $\big(1-F(a)\big)(a+\pi/2)^{-4}$ vs. $\big(1-G(a)\big)(a+\pi/2)^{-4}$ will show a clear separation. 

\medskip 
We now focus on the values $n > 3$, which we will assume when needed (though most steps will work for $n\geq 3$ or even $n\geq 2$).  

\smallskip 
Consider first the case $a\geq 0$.  In the notation from the one-sided setting, we have 
\begin{eqnarray}
\nu\left(\{ |\phi_a- \E \phi_a| \geq t\}\right) &= &\nu \left(\{ \phi_a \geq a\}\right) + \nu \left(\{ \phi_a \leq -a - 2\eta\}\right) \nonumber \\
&= &\nu ([a, \pi/2]) +  \nu ([a+2\eta, \pi/2]) . \nonumber
\end{eqnarray}
Accordingly, our problem reduces to determining whether  
\be \label{two-sided} 
\nu ([a, \pi/2]) +  \nu ([a+2\eta, \pi/2])  \stackrel{?}{\leq} e^{-n(a+\eta)^2/2}. 
\ee
If we use (for $n > 2$) the bound $\nu ([t, \pi/2]) \leq \frac 12 e^{-nt^2/2}$ (a special case of \eqref{eq:sphereconstants}), the conclusion will follow if 
\be \label{concave}
 \frac 12 \left( e^{-na^2/2} + e^{-n(a+2\eta)^2/2}\right) \stackrel{?}{\leq} e^{-n(a+\eta)^2/2}. 
\ee
Since the function $s \to e^{-ns^2/2}$ is concave on the interval $[-\frac 1{\sqrt{n}}, \frac 1{\sqrt{n}}]$, this clearly holds if $a+2\eta \leq \frac 1{\sqrt{n}}$.  A direct calculation shows that the constraint $a+\eta \leq \frac 1{\sqrt{n}}$ is also sufficient. However, this argument can not work in full generality since the function  $s \to e^{-ns^2/2}$ is convex on the interval $[ \frac 1{\sqrt{n}}, \infty)$, and so the inequality converse to \eqref{concave} holds if $a \geq \frac 1{\sqrt{n}}$. To handle such larger values of $a$, we need a strengthening of the inequality 
\eqref{eq:sphereconstants2} from Proposition \ref{prop-main}. 

Heuristically, it is clear that the bound \eqref{eq:sphereconstants2} can be improved if $s = x{\sqrt{n}}$ is large enough. Indeed, $\sqrt{n}\, \theta$ behaves roughly as a standard normal random variable $Z$ and so -- within the range of this approximation -- $\nu_n(\theta \geq x) \approx \P(Z \geq s) \sim \frac 1{\sqrt{2 \pi}\, s} e^{-s^2/2} 
= \frac 1{\sqrt{2 \pi n}\, x} e^{-nx^2/2}$ by Komatu's inequality \eqref{komatsu}, or the more precise inequality from \cite{corr} mentioned in the proof of \eqref{eq:spherelipschitz2}. Consequently, the coefficient of $e^{-nx^2/2}$ becomes small when $x{\sqrt{n}}$ is large. The same phenomenon is exemplified in the spherical case by the bound \eqref{cap-cos2}.  For our purposes, the following variant will suffice.

\begin{lemma} \label{improved}
If $n>3$, then   $ \mu(K^{x}) = \nu_n([x, \pi/2]) \leq \frac 25 e^{-nx^2/2}$ for $x\in [\frac 1{2\sqrt{n}},\frac{\pi}2]$.  For $n=3$, the inequality holds for $x\in [\frac 5{9\sqrt{n}},\frac{\pi}2]$.
\end{lemma}

Lemma \ref{improved} is based on a subtle comparison between the cosine and the exponential function. Since such ideas will also be used later in the argument, we state them separately.

\begin{lemma} \label{powers} 
We have
\begin{equation} \label{improved2} 
\cos^{n-2}\theta \geq e^{-n\theta^2/2} \quad \hbox{ for } \ \theta \in [0, 3/\sqrt{n}] \ \hbox{ and } \ n \geq 5.
\end{equation}
For $n = 3,4$, the inequality holds for $\theta \in [0, 2.67/\sqrt{n}]$ and $\theta \in [0, 2.89/\sqrt{n}]$ respectively. On the other hand, 
\begin{equation} \label{improved4}
\cos^{n-1}\theta \leq e^{-n\theta^2/2} \quad \hbox{ for } \ \theta \in [\sqrt{6/n},\pi/2]  \ \hbox{ and } \ n \geq 3.
\end{equation}
\end{lemma}

The proofs of both Lemmas involve mostly calculus, some numerics, and careful book-keeping. We relegate them to Appendices \ref{lemma13} and \ref{lemma12}. 

\medskip Returning to the proof of \eqref{two-sided}, we consider two cases. 

\smallskip \noindent {\em Case $1^\circ$} $a \geq \frac 1{2\sqrt{n}}$\quad Assuming $n>3$, both terms on the left-hand side of \eqref{two-sided} can be upper-bounded using (the first statement of) Lemma \ref{improved} and so it is enough to verify
\be \label{strong-concave}
 \frac 25 \left( e^{-na^2/2} + e^{-n(a+2\eta)^2/2}\right) \stackrel{?}{\leq} e^{-n(a+\eta)^2/2}. 
\ee
Due to the improvement in the bound for $ \mu(K^{x})$ (compared to \eqref{concave}), an argument along the lines of the proof of \eqref{eq:spherelipschitz2} will work.  First, since clearly $e^{-n(a+2\eta)^2/2}\ \leq e^{-n(a+\eta)^2/2}$, the inequality \eqref{strong-concave} can be further reduced to
\be
 e^{-na^2/2} \stackrel{?}{\leq}1.5 \ e^{-n(a+\eta)^2/2}. \nonumber
\ee
As in the proof of \eqref{eq:spherelipschitz2}, this is equivalent to 
\be \label{strong-concave2}
\eta  \stackrel{?}{\leq}  \sqrt{a^2+ \frac{2 \ln 1.5}n } -a  = \frac 1{\sqrt{n}} \left(\sqrt{u^2+ 2 \ln 1.5 } -u\right),   
\ee
where $a=u/\sqrt{n}$. \otoh from the definition \eqref{delta} of $\eta$ and appealing to Lemma \ref{improved}, we deduce that 
\be \label{strong-concave2a}
\eta = \int_a^{\pi/2} \mu(K^{x}) \leq \frac 25 \int_a^{\pi/2} e^{-nx^2/2} \,dx < \frac 2{5\sqrt{n}} \int_u^{\infty} e^{-s^2/2} \,ds . 
\ee
The last integral in \eqref{strong-concave2a} can be expressed in terms of the Gaussian error function and investigated numerically. Alternatively, as in the proof of  \eqref{eq:spherelipschitz2}, we may use the Komatu bound \eqref{komatsu}, which reduces the problem to showing that, for $u\geq 0.5$, 
 \be \label{strong-concave3}
\frac 2{5} \times \frac {2  e^{-u^2/2}}{u+\sqrt{u^2+2}}\ \stackrel{?}{\leq} \  \sqrt{u^2+ 2 \ln 1.5 } -u =  \frac{2 \ln 1.5}{u+\sqrt{u^2+ 2 \ln 1.5 }} .
\ee
Siince \ $u+\sqrt{u^2+2} \geq u+\sqrt{u^2+ 2 \ln 1.5 }$, these two denominators can be discarded.  To complete the argument, it remains to verify the resulting inequality at $u= 0.5$.  (The inequality  \eqref{strong-concave3} actually holds for all $u\geq 0$, but showing that is not needed.)

\smallskip \noindent {\em Case $2^\circ$} $a \leq \frac 1{2\sqrt{n}}$\quad  In this case we can not use Lemma \ref{improved} to estimate $\eta$,  but -- as in the proof  of \eqref{eq:spherelipschitz2} --  the weaker bound \eqref{eq:sphereconstants2}  combined with Komatu's inequality \eqref{komatsu} will be -- for different reasons -- sufficient. In the notation from Case $1^\circ$, we have 
\[
\eta \leq \frac{1}{2} \int_a^{\pi/2} e^{-nx^2/2} \, dx < \frac{1}{2\sqrt{n}} \int_u^{\infty} e^{-s^2/2} \, ds < \frac{1}{\sqrt{n}} \times \frac {e^{-u^2/2}}{u+\sqrt{u^2+2}}.
\]
Consequently, 
\[
a + \eta <  \frac{1}{\sqrt{n}} \times \left(u + \frac {e^{-u^2/2}}{u+\sqrt{u^2+2}} \right) .
\]
It is readily verified that the expression in the parentheses is less than $1$ for $u\leq \frac 12$.  
In particular, for $0\leq u = a\sqrt{n} \leq \frac 12$, we get $a +\eta <  \frac{1}{\sqrt{n}}$, and so we are in the range of applicability of \eqref{concave}. 
(The argument is almost as clean if we use the more precise expression involving the Gaussian error function;  it yields the bound $a +\eta <  \frac{1}{\sqrt{n}}$ for $u\leq 0.69$.) 

Finally, let us recall that when $n=3$, the inequality \eqref{two-sided} was verified numerically, and -- in the range $a\geq 0$ -- it was never close. Alternatively, the general argument presented above can be easily patched up when specified to the instance $n=3$ (and only {\em Case $1^\circ$} requires patching).

\medskip It remains to handle the case $a<0$.  As in the context of the one-sided bound, it is then more transparent to rewrite the formula for $t$ as
\be\label{2sidedneg1}
 t=   \int_{b}^{\pi/2} \mu(K^x) \, dx .
\ee
where $b=-a \in [0,\pi/2]$ (see \eqref{eq:tAlt2}), while the inequality to be verified  (cf. \eqref{eq:tAlt3}) becomes
\be \label{2sidedneg2}
\mu(K^{b})-  \mu(K^{b+2t})\stackrel{?}{\geq} 1- e^{-nt^2/2} .
\ee

\smallskip \noindent  As for $a\geq 0$, we will consider separately the cases when $b$ is ``small'' and ``not-so-small.''

\medskip\noindent  {\em Case $1^\circ$,  small $b$}  : \ \ 
If $b$ is sufficiently small (to be made precise later), $\theta = b+2t$ will be within the range of applicability of inequality \eqref{improved2}, and the approach from (the proof of) Lemma \ref{improved} will work.    
Specifically, we can deduce  then that 
\be \label{2sidedneg2a}
\mu(K^{b})-  \mu(K^{b+2t}) = (2I_{n-2})^{-1}\int_b^{b+2t} \cos^{n-2}\theta \, d\theta \geq 
(2I_{n-2})^{-1}\int_b^{b+2t} e^{-n\theta^2/2} \, d\theta .
\ee
Substituting $s=\sqrt{n}\,\theta$, the inequality \eqref{2sidedneg2} reduces to
\be \label{2sidedneg3}
(2I_{n-2})^{-1} \times \frac 1{\sqrt{n}} \int_u^{u+2v} e^{-s^2/2} \, ds  \stackrel{?}{\geq} 1- e^{-v^2/2} ,
\ee
where $u = b\sqrt{n}$ and $v = t\sqrt{n}$.  At the same time,  as in Eqs. \eqref{eq:tAlt2}-\eqref{eq:tAlt6}, and in view of Lemma \ref{v_k-v_k+1} and Proposition \ref{prop-main}, 
\be \label{2sidedneg4}
v =v(b,n) \leq \sqrt{n}\, I_{n-1} \mu_{n+1}(K^b) \leq \frac { \sqrt{n}\, I_{n-1} }2 \, e^{-u^2/2}.
\ee
For future reference, let us note that \eqref{2sidedneg4} and Proposition \ref{wallis} imply immediately that $v(b,n)\leq  \frac { \sqrt{2}\, I_{1} }2=1/\sqrt{2} <1$. 

To summarize, we need to show that if $v$ satisfies the constraint \eqref{2sidedneg4}, then \eqref{2sidedneg3} holds for $u$ in the appropriate range.  Furthermore, since (for fixed $u$), $v\to \frac 1v \int_u^{u+2v} e^{-s^2/2} \, ds$ is decreasing, while $v \to  \frac 1v  \big(1- e^{-v^2/2}\big)$ is increasing (for $v\leq 1$), it is enough to consider the largest possible value of $v$, e.g., $v=\frac { \sqrt{n}\, I_{n-1} }2 \, e^{-u^2/2}$. Thus the problem is reduced to comparing two functions of $u$, which depend rather weakly on $n$ since, by Proposition \ref{wallis}, the coefficients  appearing in them satisfy  
$(2I_{n-2})^{-1} \times \frac 1{\sqrt{n}}  \to \frac 1{\sqrt{2\pi}}$ and $\frac { \sqrt{n}\, I_{n-1} }2 \to \sqrt{\frac {\pi}8}$. This suggests verifying first the asymptotic version of the statement, namely : 
\be \label{2sidedneg5}
v =  \sqrt{\frac {\pi}8} \, e^{-u^2/2} \stackrel{?}{\implies}  \frac 1{\sqrt{2\pi}} \int_u^{u+2v} e^{-s^2/2} \, ds \geq 1- e^{-v^2/2}. 
\ee
 A numerical check shows that this statement ``comfortably'' holds in the relevant $u$-range (say, $0\leq u\leq 3$), see Figure \ref{asymptotic}. 
 \begin{figure}[ht!]
\includegraphics[width=0.7\textwidth]{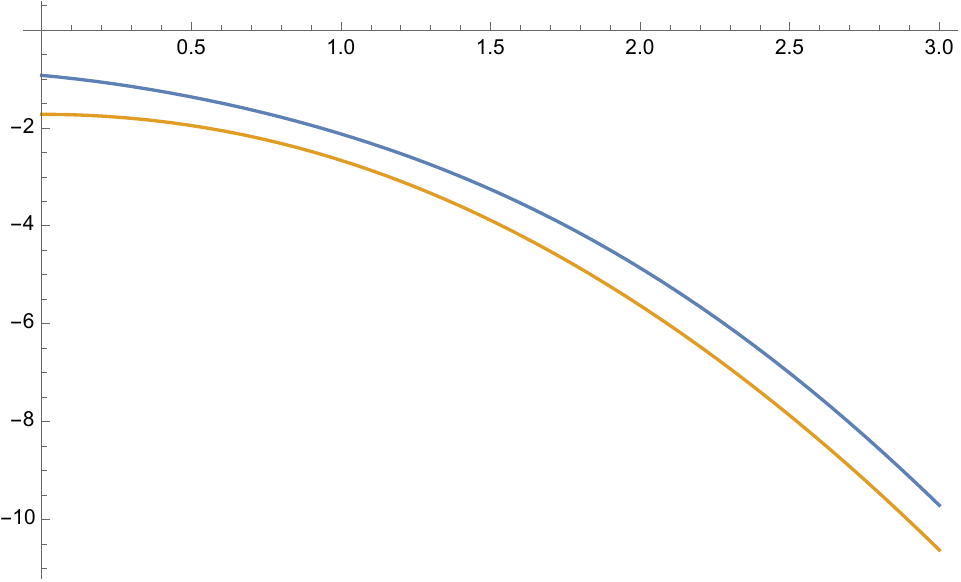}
\caption{Verification of \eqref{2sidedneg5}. The top and the bottom graph depict the logarithms of (respectively) the left and the right hand side of the inequality in \eqref{2sidedneg5}.}
\label{asymptotic}
\end{figure}
This implies that the inequality \eqref{2sidedneg3} holds under the constraint \eqref{2sidedneg4} if $n$ is large enough.  Moreover, since the sequences of coefficients are monotone (by Proposition \ref{wallis}), once we determine that  \eqref{2sidedneg3} holds for $n=n_0$,  it will follow that it is also valid for $n>n_0$. A direct check shows that this works for $n_0=3$, which is enough for our purposes. 

\smallskip To complete this part of the argument, we need to determine the range of the parameter $u$ that assures that 
\eqref{improved2} can be applied with $\theta = b+2t$ (that is, 
$b+2t \leq 3/\sqrt{n}$ or, equivalently, $u+2v \leq 3$ if $n\geq 5$, and similarly for $n=3,4$).  Using the bound \eqref{2sidedneg4} with $n=3$, we see that $u+2v \leq 2.67$ if $u\leq 2.52$. Since the sequence $(\sqrt{n}\,I_{n-1})$ decreases by Proposition \ref{wallis}, it follows that the same is true for $n>3$.  In other words, the method from  {\em Case $1^\circ$} works for $u\leq 2.52$. This is amply sufficient as the approach from {\em Case $2^\circ$} will cover the range  
$u\geq \sqrt{6}$, and $\sqrt{6} \approx 2.45 < 2.52$. 


\medskip\noindent  {\em Case $2^\circ$, not-so-small $b$}  : \ \ To handle the values of $u>\sqrt{6}$, which correspond to $b > \sqrt{6/n}$, we need a more specialized upper bound for $t$.  To that end, we  appeal to \eqref{cap-cos2}, which restated in the current setup becomes
\be \label{2sidedneg7}
\mu_n(K^x) \leq  (\sqrt{2\pi}\, \kappa_n \sin x)^{-1} \cos^{n-1} x .
\ee
Accordingly 
(cf. \eqref{2sidedneg1}) 
\begin{eqnarray}
t &\leq&\int_b^{\pi/2} (\sqrt{2\pi}\, \kappa_n \sin x)^{-1} \cos^{n-1} x \, dx \nonumber \\
&\leq &   (\sqrt{2\pi}\, \kappa_n \sin b)^{-1} \int_b^{\pi/2} \cos^{n-1}x \, dx \nonumber \\
&=&  (\sqrt{2\pi}\, \kappa_n \sin b)^{-1} \times 2I_{n-1} \, \mu_{n+1}(K^b) \nonumber \\
&\leq& \frac{I_{n-1}}{\pi  n \sin^2 b} \times \cos^n b,  \label{2sidedneg8}
\end{eqnarray}
where in the last inequality  we used (again) \eqref{2sidedneg7} and the identity $\kappa_n \kappa_{n+1} = n$. 

\smallskip We now argue similarly as in the one-sided context. The left-hand side of \eqref{2sidedneg2} will be lower-bounded by $2t \times (2I_{n-2})^{-1}  \cos^{n-2}(b+2t)$ (cf.\ \eqref{2sidedneg2a}) and the right-hand side upper-bounded by $nt^2/2$, which reduces the problem to 
\be \label{2sidedneg6}
 \cos^{n-2}(b+2t) \stackrel{?}{\geq}  I_{n-2} \times \frac{nt}2 .
\ee

Appealing to \eqref{2sidedneg8} allows further reduction to 
\be \label{2sidedneg9}
t  \leq \frac{I_{n-1}}{\pi  n \sin^2 b} \times \cos^n b \quad  \stackrel{?}{\implies} \quad \cos^{n-2}(b+2t) \geq   \frac{\cos^n b}{4(n-1)  \sin^2 b} , 
\ee
where we used $I_kI_{k+1}=\frac{\pi}{2(k-1)}$. 

As in the argument that led to \eqref{2sidedneg5}, let us consider first an asymptotic version of \eqref{2sidedneg9}. That is, substitute  $u = b\sqrt{n}$ and $v = t\sqrt{n}$ and let $n\to \infty$,  which leads to 
 \be \label{2sidedneg10}
u  \geq \sqrt{6} \ \hbox{ and } \ v  \leq \frac{e^{-u^2/2}}{\sqrt{2 \pi} u^2 } \quad  \stackrel{?}{\implies} \quad e^{-(u+2v)^2/2} \geq   \frac{e^{-u^2/2}}{4u^2} . 
\ee
To establish \eqref{2sidedneg10}, it is clearly enough to assume equality in  the constraint on $v$, and it is then apparent (numerically) that the inequality on the right comfortably holds. In fact, it does hold for $u\geq 0.84$ and, for $u \geq \sqrt{6}$, the ratio of the two sides is greater than $23$.  We can not, however, deduce immediately that \eqref{2sidedneg9} holds for sufficiently large $n$ since there is no obvious monotonicity with respect to $n$ and we do not know if the convergence involved in obtaining  \eqref{2sidedneg10} is appropriately uniform. Still, patching the calculation is rather routine; we sketch the main points in Appendix 
\ref{eq60}. \qed

\subsection{The Gaussian case : Proof of Corollary \ref{cor:gaussian}} 
  The Gaussian isoperimetric inequality \cite{Borell, ST} reduces the problem to $n=1$.  
The obvious line of argument is now to invoke some version of the Poincar\'e Lemma  (appropriately normalized marginals of $\mu= \mu_n$ converge, as $n\to \infty$, to the  normal distribution) and then appeal to  \eqref{eq:spherelipschitz3}, but there are some minor technical issues that need to be addressed.  First,  $\Theta_n$,   the random variable distributed according to $\nu_n$, is not exactly a marginal of $\mu_n$ (the marginals are parametrized by $\sin \theta$ rather than by $\theta$).  Next, we have to make sure that the convergence of  $\sqrt{n} \Theta_n$ to the standard normal $Z$ preserves probabilities and moments.  An elementary way to resolve these issues is to consider the density of  $\sqrt{n} \Theta_n$,  which is (on its support) $g_n(\theta):=(2 I_{n-2} \sqrt{n})^{-1} \cos^{n-2}(\theta/\sqrt{n})$.  Once we take into account the properties of $I_m$ stated in Proposition \ref{wallis}, it is an elementary exercise to show that $g_n(\theta) \to (2\pi)^{-1/2}e^{-\theta^2/2}$ (the density of $Z$), and that this convergence is dominated in a rather strong sense:  we have $0\leq g_n(\theta) \leq (2\pi)^{-1/2} e^{-\frac{n-2}{2n} \theta^2} \leq  (2\pi)^{-1/2 } e^{-\theta^2/6}$ for all $\theta$ and $n\geq 3$.   The dominated convergence theorem implies then the convergence of all probabilities and all moments. \qed

\medskip An alternative line of argument is to appeal to Lemma \ref{extreme2} and then compare $\P(|Z - \E Z| > t)$ to $e^{-t^2/2}$, where $t=t(a) = a+ \int_a^{\infty} \gamma_1\big([x,\infty)\big) \, dx$. Since all these quantities can be expressed in terms of the Gaussian error function, there is no problem with a numerical verification, see Figure \ref{gaussian}.  For complete rigor, this should be accompanied by an asymptotic analysis as $a \to \pm \infty$.

 \begin{figure}[ht!]
\includegraphics[width=0.75\textwidth]{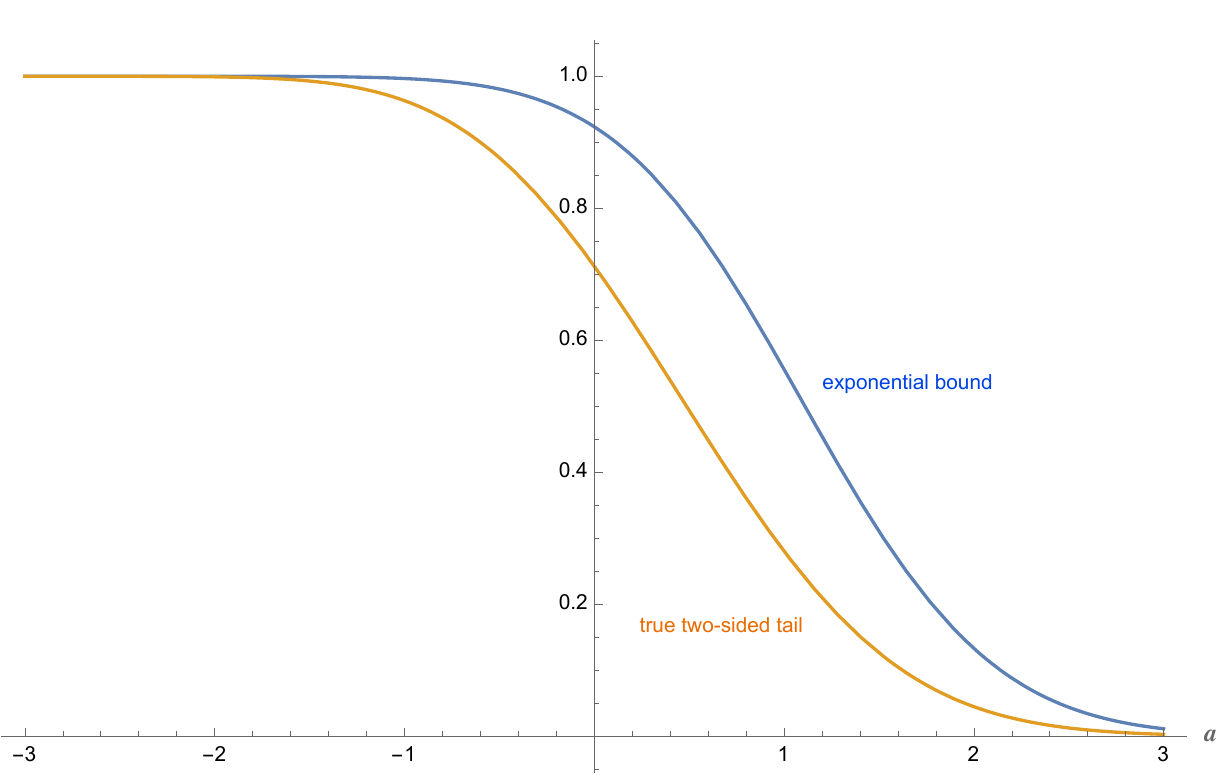}
\caption{The Gaussian case : the  comparison of the true two-sided tail given by the extremal functions $\phi_a$ and the asserted exponential bound (both as a function of $a$). \label{gaussian}}
\end{figure}
Finally, one could redo in the Gaussian setting the rather rigorous calculations that were performed in the spherical case. These would be substantially easier since we do not need to worry about the dependence on the dimension. As a matter of fact, we did some such calculations to provide heuristics for the spherical case. However, an argument of that nature wouldn't be pretty.  It would be good to have a neat proof based on standard properties of the  Gaussian error function, perhaps along the lines of \cite{corr} or the follow-up papers \cite{Kouba, RW}.

\section{Products of spheres} \label{products}

In this section we will discuss perspectives for improving isoperimetric/concentration inequalities on $\big(S^{n-1}\big)^k$ for $k >1$.  The discussion is somewhat exploratory in nature, with some results based on numerics and many estimates presumably not optimal, and is intended to encourage further research.

As is well known, Fact \ref{fact:MS} generalizes to products $S^{n+1} \times S^{n+1} \times \ldots S^{n+1}$ with arbitrary number of factors (see, e.g., \cite{MS86}, section 6.5.2; note, however, that he family discussed there should involve $S^{n+1}$ and not $S^{n}$). This is because Ricci curvature $R(S^{n+1})$ is $n$, and consequently the same is true for the product and one may apply the following comparison result due to Gromov (see section 6.4 and Appendix I in \cite{MS86}). 

\begin{fact}\label{fact:Gromov} 
Let $X$ be a an $m$-dimensional Riemannian manifold, whose Ricci curvature $R(X)$ is bounded from below by $\kappa >0$. Choose $r>0$ so that $R(rS^m)=(m-1)/r^2 = \kappa$. 
Denote by  $\mu_X$  the  normalized Riemannian measure on $X$ and by $\mu$ the normalized Haar (surface) measure on the sphere $rS^m$.  Next, let $A\subset X$,  $t>0$ and $B\subset rS^m$ be a cap such that $\mu_X(A)=\mu(B)$. Then $\mu_X(A_t)\geq\mu(B_t)$. 
\end{fact}
As earlier, $B_t$ is again a cap and so its volume can be -- after rescaling by $r$ -- expressed as an integral of the type \eqref{cap}. 
In particular, if $\mu_X(A)=\frac 12$, then 
\[
\mu_X(A_t)\geq 1- \frac{\int_{t/r}^{\pi/2} \cos^{m-1}\theta \, d\theta}{2I_{m-1}} = 1- q_{m+1}(t/r) , 
\]
where $q_{m+1}(\cdot)$ is defined by \eqref{qn}. Specifying further to $X=\big(S^{n-1}\big)^k$ (with $n>2$), whose Ricci curvature is $\kappa=n-2$, we are led to 
\be \label{rescale}
m=(n-1)k \quad  \hbox{ and } \quad r=\sqrt{\frac{m-1}{\kappa}}=\sqrt{\frac{(n-1)k-1}{n-2}},
\ee 
and -- after appealing to the bound \eqref{eq:sphereconstants} from Theorem \ref{sphereconstants}  --  to
\bprop \label{prop-k}
Let $n >2$, $k\geq 2$, and let $\sigma$ be the normalized product measure on $\big(S^{n-1}\big)^k$.  Next,  let $A \subset \big(S^{n-1}\big)^k$ be such that  
$\sigma(A) \geq \frac 12$. If $t>0$, then 
\be \label{eq:prop-k}
\sigma(A_t) \geq  1 -  \frac 12 \; q_{m+1}(t/r)  \geq  1 - \frac 12 \exp\Big(-\frac{\big((n-1)k+1\big)(n-2)}{(n-1)k-1}\, \frac{t^2}{2}\,\Big) . 
\ee
\eprop
Since the fraction inside $\exp$ is clearly greater than $n-2$, it follows that $\sigma(A_t) \geq 1-  \frac{1}{2}\,e^{-(n-2)t^2/2}$, which is slightly better than the estimate from \cite{MS86} mentioned at the beginning of this section: the multiplicative constant $\frac 12$ instead of $\sqrt{\pi/8}$ (this is because we are using Theorem \ref{sphereconstants}  and not Fact \ref{fact:MS}). There is also a slight improvement in the coefficient of $t^2$ in the exponent, but it becomes less and less significant as $k$ increases. In order to be able to deduce the same bound as in Theorem \ref{sphereconstants} for $k\geq 2$, we need a stronger version of \eqref{qn} (or, equivalently, of \eqref{eq:sphereconstants} or \eqref{eq:sphereconstants2})  with an appropriate ``excess'' in the exponent to compensate for the coefficient of $\frac{t^2}2$  in \eqref{eq:prop-k} being strictly smaller than $n$.  
Specifically, define
\begin{equation}\label{qnk}
q_{n,\xi}(x):= \frac{\int_x^{\pi/2} \cos^{n-2}\theta \, d\theta}{ I_{n-2}} e^{(n+\xi)x^2/2} = q_{n}(x) e^{\xi x^2/2} .
\end{equation}
and suppose that the inequality 
\be \label{excess}
q_{m+1,\xi}(x)  \leq 1
\ee
is valid with $m=(n-1)k$ and with $\xi$ such that $\frac{\big((n-1)k+1+\xi\big)(n-2)}{(n-1)k-1} = n$. Then, repeating {\em mutatis mutandis} the argument that led to Proposition \ref{prop-k}, we obtain -- for this particular choice of $n, k$, and for the respective values of $t$ -- an improvement to \eqref{eq:prop-k} with  $1 - \frac 12 \exp\big(-nt^2/2\big)$ on the right hand side, which is a much neater expression. 

As it turns out, as speculative as the bound \eqref{excess} appears, it is not unreasonable. 
Of course, as already pointed out in the Remark following  the proof of Theorem \ref{sphereconstants}, 
we cannot expect it to hold for all $m$ and all $x$, but it may conceivably be true for $m\geq m(\xi)$. 
Below is an analysis showing that validating \eqref{excess} is actually quite feasible. 

We note first that the equation $\frac{\big((n-1)k+1+\xi\big)(n-2)}{(n-1)k-1} = n$  resolves to 
$\xi = 2(k-1)\frac{n-1}{n-2}$ or $\xi = 2(k-1)\frac{m}{m-k}$. The last two expressions are decreasing functions of respectively $n$ or $m$, which means that the threshold value for the excess $\xi$ that is sufficient for our purposes can be chosen as a function of $k$ only. Next, it follows immediately from the definition \eqref{qnk} of  $q_{n,\xi}$ and from \eqref{n+2n} that $q_{n+2,\xi}(x)  \leq q_{n,\xi}(x)$ for all $n, \xi$ and $x$. 
Thus, for each $x\in[0,\pi/2]$ and for all $\xi\geq0$,  
the sequences $(q_{2n,\xi}(x))$ 
and $(q_{2n+1,\xi}(x))$ are nonincreasing. This means that, given $\xi \geq 0$, once we establish \eqref{excess} for certain $m_0$ and $m_0+1$, it will be valid for all $m\geq m_0$, and consequently for all sufficiently large $n$ (with the last qualification depending on $k$). 

A numerical check indicates  that $q_{4,1}(x) \leq 1$ and $q_{6,2}(x) \leq 1$ for $x\in[0,\pi/2]$, and that such inequalities do not generally hold for $q_{3,1}$ and $q_{5,2}$. This suggests 
that, in the setting of Theorem \ref{sphereconstants}, the 
 bound $\mu(A_t )\geq 1- \frac{1}{2}e^{-t ^2(n+1)/2}$ is valid for $n\geq 4$ 
 and the bound $\mu(A_t )\geq 1- \frac{1}{2}e^{-t ^2(n+2)/2}$ is valid for $n\geq 6$. 
 It would be interesting to rigorously determine the threshold values $n=n(\xi)$,
 in addition to the numerical results indicated above and further explored below. 

As a demonstration, let us focus on the instance $k=2$. A numerical check using {\em Mathematica} shows that in that case:  

\smallskip
$\bullet$ if $n=3$, hence $m=4$ and $\xi = 4$, then \eqref{excess} holds for $x \not\in (0.47595,1.45105)$; taking into account the rescaling $x=t/r$ we deduce (cf. \eqref{rescale}) that, in the setting of Proposition \ref{prop-k}  the bound $\sigma(A_t) \geq  1 -  \frac 12 e^{-nt^2/2}$ holds for 
$t\leq 0.82437$

\smallskip
$\bullet$ if $n=4$, hence $m=6$ and $\xi = 3$, then \eqref{excess} and all the subsequent bounds hold for $x \not\in (0.71556,1.19952)$,  so (after rescaling) we deduce  that   $\sigma(A_t) \geq  1 -  \frac 12 e^{-nt^2/2}$  for  $t\leq 1.1314$; however, if we use the chordal distance instead of the geodesic distance, then the bounds hold in the entire range 

\smallskip
$\bullet$ if $n=5$, hence $m=8$ and $\xi = 8/3$, then \eqref{excess} and all the subsequent bounds hold in the entire respective range 

\smallskip
$\bullet$ if $n>5$, the same holds by monotonicity; note that for $k=2$ all $m$'s are even, and so our inductive ``initialization'' requires verifying only one value $m_0$. 

\smallskip
The above considerations can be summarized in the following statement. 
\bt
Let $\sigma$ be the normalized product measure on $\big(S^{n-1}\big)^2$ and let $A \subset \big(S^{n-1}\big)^2$ be such that  
$\sigma(A) \geq \frac 12$. If $t\geq 0$ and $n>4$, then 
\[
\sigma(A_t) \geq  1 -  \frac 12 e^{-nt^2/2} . 
\] 
If $n=3$ and $n=4$, the above bound holds for, respectively, $t\leq 0.82437$ and $t\leq 1.1314$. Additionally, the bound holds in the entire range if $n=4$ and if the enlargements $A_t$ are defined via the chordal distance rather than the geodesic distance. 
\et

For $k=3$, the bound \eqref{excess} holds for $n>6$ and for $n=6$ with the chordal distance; for $k=4$ for $n>7$ and for $n=6,7$ with the chordal distance. So it appears that the threshold for allowable dimensions $n$ increases with $k$, possibly unboundedly.  
Again,  it would be interesting to rigorously determine the dependence of the allowable range of $n$ as a function of $k$ (assuming it indeed does not stabilize, which would be a desirable property, but not very likely in view of the above numerical results).  Another useful -- and perhaps not that hard -- result would be a good universal lower bound on $b$ such that the bounds hold for $x\in [0,b]$, or for $t\in [0,b]$. Numerics suggest that those intervals are never {\em really} small. 

The final remark is that -- unlike in the case $k=1$ -- we do not have an exact calculation, based on the knowledge of extremal subsets, but one that relies on the Gromov's comparison theorem (Fact \ref{fact:Gromov}). 
While there are many very sophisticated approaches to isoperimetric problems on product spaces (e.g. \cite{Schechtman,talagrand}),  we are not aware of the {\em precise} solution to the problem even in the case of the torus $(S^1)^k$.  So it is possible, and quite likely,  that in reality the bounds hold for a larger set of parameters than  what follows from the argument above. It may be feasible to test, e.g., $k=2$ and $n=3$ or $n=4$ by looking at some specific sets $A$ and the (relatively large) values of $x$ or $t$ suggested by the numerics leading to the results described above. 

{\small \vskip6mm \noindent {\bf Acknowledgements.} GA was supported in part by ANR (France) under the grant ESQuisses (ANR-20-CE47-0014-01). 
The research of JJ and SJS has been supported in part by grants from the {\itshape National Science Foundation (U.S.A.)}.}

\section{Appendix }

\subsection{Achievability of tail estimates in Lemma \ref{extreme}} \label{achieve}

Here we sketch a proof of the assertion from the Remark following the proof of Lemma \ref{extreme}, 
which said that the supremum in \eqref{opt} is attained for some $\psi=\phi_a$ and, moreover, 
for $a$ verifying $t= a-\E \phi_a$.  In view of \eqref{opt2}, this reduces to the analysis of the family $\{\phi_a : a\in \R\}$ 
without having to consider a general $\psi \in \mathcal{L}$. We will start with a series of elementary observations.

\smallskip \noindent $1^\circ$   Since $\phi_a \leq a$, it follows that $\E \phi_a \leq a$, with equality iff  $\nu$ is supported on $[a,\infty)$. 

\smallskip \noindent $2^\circ$   Since $\|\phi_a-\phi_b\|_\infty = |a-b|$ and $\phi_a\leq \phi_b$ if $a\leq b$, the function $a \to \E \phi_a$ is $1$-Lipschitz and non-decreasing. 
Consequently, the same is true for $a \to \E \phi_a +t$. 

\smallskip \noindent $3^\circ$   Similarly, $a \to a-\E \phi_a$ is nondecreasing; this follows from  $a-\E \phi_a = \int (a-x)^+ d\nu(x)$ and $a\to (a-x)^+$ being a nondecreasing function of $a$ (for each fixed $x$). 

\smallskip \noindent $4^\circ$  If $t>0$, then the supremum in \eqref{opt} is strictly smaller than $1$. 
 The hypothesis $M:= \int |x|\, d\nu(x) < \infty$  implies that  $\forall \ep >0 \ \exists \delta >0$ such that if $\nu(A)< \delta$, then $\int_A |x| d\nu(x) < \ep$. 
Let $\psi \in \mathcal{L}$ and suppose (as we can) that $f(0)=0$, which implies $|\psi(x)|\leq |x|$. 
Denote $A=U(\psi,t)^c$; our objective is to show that $\nu(A)$ can not be too small.  We have
$$
\E \psi = \int_{U(\psi ,t)} \psi \, d\nu + \int_A \psi\, d\nu \geq \nu\big(U(\psi ,t)\big)(\E \psi + t) -\int_A |\psi |\, d\nu ,
$$
which can be rewritten as  
$$
\nu(A)\, \E \psi    \geq \big(1-\nu(A)\big) t  -  \int_A |\psi |\, d\nu.
$$
Now, set $\ep = \frac t4$ and choose the corresponding $\delta \in (0, \frac 12)$. If  $\nu(A) < \delta$, then $\big(1-\nu(A)\big) t  -  \int_A |\psi |\, d\nu > \frac t2 - \ep = \frac t4$, while 
$\nu(A) \E \psi   \leq \nu(A) \E |\psi | \leq \nu(A) M$ and so
$$
\frac t4 < \nu(A) M .
$$
In other words, either  $\nu(A) \geq \delta$, or  $\nu(A) \geq \frac t{4M}$, so $\nu(A) \geq \min\{\delta,\frac t{4M}\} > 0$, as asserted. \\
{\em Note} : The assertion $4^\circ$ will follow independently from the other observations, but we include it here since the argument works for any $1$-Lipschitz function $\psi $ and not just for $\psi \in \mathcal{L}$.

\smallskip \noindent $5^\circ$  $\lim_{a\to +\infty} \E\phi_a = \int x \, d\nu(x)$ and $\lim_{a\to -\infty} a- \E\phi_a=0$. Both of these follow from $\int |x| \, d\nu(x) < \infty$ via the dominated convergence theorem. 

\smallskip With this preparation, the conclusion is very easy.  Denote $\xi(a) := \E\phi_a+t$, $a\in \R$. By  $2^\circ$, the function $\xi$ is $1$-Lipschitz and non-decreasing. By $5^\circ$, it has an oblique asymptote $\ell_-(a)=a+t$ as $a\to -\infty$ and a horizontal asymptote $\ell_+(a)=\int x \, d\nu(x)+t$ as $a\to +\infty$.  Since, by $3^\circ$, $\xi(a)-a$ is nonincreasing, it follows that there is a unique value $a_0$ such that \\
$\bullet$ \ $\xi(a) > a$ if $a< a_0$\\
$\bullet$ \ $\xi(a) \leq a$ if $a \geq a_0$, with equality if $a= a_0$. \\
Let us decode what these inequalities mean. First, $\xi(a) > a$ means $\E\phi_a+t > a \geq \phi_a$; in that case $U(\phi_a,t)=\{\phi_a \geq \E\phi_a+t\} = \emptyset$. \otoh if $\xi(a) =\E\phi_a+t \leq a$, then $U(\phi_a,t)= [\xi(a), \infty)$. Since, again, $\xi$ is non-decreasing by $2^\circ$, the largest value of $\nu( [\xi(a), \infty))$ will be attained for the smallest value of $a$ for which 
$\xi(a) = a$, i.e., for $a=a_0$, and the condition $\xi(a) = a$ means precisely $\E\phi_a+t = a$ or $t= a-\E \phi_a$. \qed

\subsection{ Proof of Lemma \ref{extreme2}} \label{lemma11}

 We fix $t>0$ and proceed in several steps.  

\smallskip \noindent {\it Step $1^\circ$} First, there is the essentially trivial observation that -- from the point of view of estimating $\lambda(\psi,t):= \nu(\{|\psi - \E \psi | \geq t\})$ -- the function $x\to f(x)$ is equivalent to $x\to f(x)+c$  (for any $c\in \bbR$) and, due to the symmetry of $\nu$, to $x\to -f(-x)$.  

\smallskip \noindent {\it Step $2^\circ$}  Second, extremal functions (i.e., such that $\lambda(\cdot,t)$ is maximal) do exist.  Indeed, suppose that $(\psi _k)$ is a sequence of functions for which $\lambda(\psi_k,t)\to \max_{\psi \in \mathcal{L}} \lambda(\psi,t)=: \Lambda$.  By the previous remark, we may assume that $\psi_k(0)=0$ for all $k$, and it then follows that there is a subsequence $(\psi_{k_i})$ of $(\psi _k)$ converging uniformly on bounded intervals\footnote{In our setting, the measures $\nu=\nu_n$ are all supported on $[-\pi/2,\pi/2]$, which makes the argument even more straightforward.} 
 to some function $\psi$, which necessarily belongs to $\mathcal{L}$.  Since $\int |x| d\nu(x) < \infty$, it follows from the dominated convergence theorem that $\E \psi_{k_i} \to \E \psi$. 
This implies that $\{|\psi  - \E \psi  | \geq t\} \supset \limsup_i \{| \psi_{k_i}  - \E \psi_{k_i}| \geq t\}$ and, consequently, that $\lambda(\psi,t) \geq \limsup_i\lambda(\psi_{k_i},t) = \Lambda$.  So the limit function $\psi$ is extremal.

\smallskip \noindent {\it Step $3^\circ$}   The next observation is slightly less trivial; it gives the first hint why functions of the form \eqref{phi_a}  may be extremal. Let $\psi \in \mathcal{L}$ and let $\alpha$ be defined by  $\psi(\alpha) = \E \psi$ (it exists by the intermediate value theorem,  see Figure \ref{psilambda} for this and the subsequent steps).  
We claim that if $\psi$ is extremal, then (at least in all cases of interest, to be clarified later)
\be \label{linear}
\psi(\theta) - \psi(\alpha) = \theta -\alpha \ \hbox{ for }  \ \theta \in [\alpha-t,\alpha+t]. 
 \ee
 If this is not the case, then, denoting $\theta_0 := \max\{ \theta : \psi(\theta)  = \E \psi -t \}$  and 
 $\theta_1 := \min\{ \theta : \psi(\theta)  = \E \psi +t \}$, we will have $\theta_1-\theta_0 > 2t$. We now set $\alpha_0:= \theta_0+t, \alpha_1:= \theta_1-t$  (so $\alpha_0<\alpha_1$) and define $\psi_0$ and $\psi_1$ as follows
 
  \bes
 \psi_0(\theta) = \left\{  \begin{array}{ll} \E \psi  +\theta-\alpha_0&{\rm if } \ \theta \in [\alpha_0-t,\alpha_0 +t]\\
\E \psi +t  &{\rm if } \  \theta \in [\alpha_0 +t, \theta_1]\\
\psi(\theta) &{\rm otherwise } 
\end{array} \right.\hskip-3mm, \ 
\  \psi_1(\theta) = \left\{ \begin{array}{ll}  
 \E \psi -t  &{\rm if } \  \theta \in [\theta_0 , \alpha_1-t]\\
\E \psi +\theta-\alpha_1&{\rm if } \ \theta \in [\alpha_1-t, \alpha_1+t]\\
\psi(\theta) &{\rm otherwise } 
\end{array} \right. 
 \ees 

\bigskip

\begin{figure}[ht!]
\centering
    \begin{tikzpicture}
    \coordinate (A) at (0,2);
    \coordinate (B) at (3,2);
    \coordinate (C) at (4.8,2);
    \coordinate (D) at (2,5);
    \coordinate (E) at (5,5);
    \coordinate (F) at (6.8,5);
    \coordinate (AD) at (1,3.5);
    \coordinate (BE) at (4,3.5);
    \coordinate (CF) at (5.8,3.5);

    \coordinate (l1) at (-1,0);
    \coordinate (r1) at (8,0);
    \coordinate (l2) at (-1,2);
    \coordinate (r2) at (8,2);
    \coordinate (l3) at (-1,3.5);
    \coordinate (r3) at (8,3.5);
    \coordinate (l4) at (-1,5);
    \coordinate (r4) at (8,5);
    
    \draw[dotted] (l1) -- (r1);
    \draw[dotted] (l2) -- (r2);
    \draw[dotted] (l3) -- (r3);
    \draw[dotted] (l4) -- (r4);
    \node[right] at (r2) {$\E \psi -t$};
    \node[right] at (r3) {$\E \psi = \psi(a) = \E \psi_s$};
    \node[right] at (r4) {$\E \psi + t$};
    
    \draw[blue, very thick] (A) -- (B) -- (C) -- (F);
    \draw[green, very thick] (B) -- (E);
    \draw[red, very thick] (A) -- (D) -- (E) -- (F);
    
    \draw[dashed] (A) -- ++(0,-2);
    \node[below] at (0,0) {$\theta_0$};
    \draw[dashed] (AD) -- ++(0,-3.5);
    \node[below] at (1,0) {$\alpha_0$};
    \node at (2.1,3.5) {$\bullet$};
    \draw[dashed] (2.1,3.5) -- ++(0,-3.5);
    \node[below] at (2.1,0) {$\alpha$};
    \draw[dashed] (B) -- ++(0,-2);
    \node[below] at (3,0) {$\alpha_s\!-\!t$};
    \draw[dashed] (BE) -- ++(0,-3.5);
    \node[below] at (4,0) {$\alpha_s$};
    \draw[dashed] (E) -- ++(0,-5);
    \node[below] at (5,0) {$\alpha_s+t$};
    \draw[dashed] (CF) -- ++(0,-3.5);
    \node[below] at (5.8,0) {$\alpha_1$};
    \draw[dashed] (F) -- ++(0,-5);
    \node[below] at (6.8,0) {$\theta_1$};

    \draw plot [smooth] coordinates {(-1,1) (A) (2.1,3.5) (2.5,4) (4,4.2) (F) (8,6)};

    \draw[red,->] (0.8,4.2) -- (1.2,4);
    \node[left,red] at (0.9,4.3) {$\psi_0$};

    \draw[green,->] (3.1,3) -- (3.5,3);
    \node[left,green] at (3.2,3) {$\psi_s$};

    \draw[blue,->] (6.3,3) -- (5.9,3.2) ;
    \node[right,blue] at (6.2,3) {$\psi_1$};

    \node[above] at (7.5,5.5) {$\psi$};   

    \end{tikzpicture}
\caption{The functions  $\psi$, $\psi_0$,  $\psi_1$, and $\psi_s$. Note that $\a_0=\theta_0+t$ and $\theta_1=\a_1+t$. }
\label{psilambda}
\end{figure}
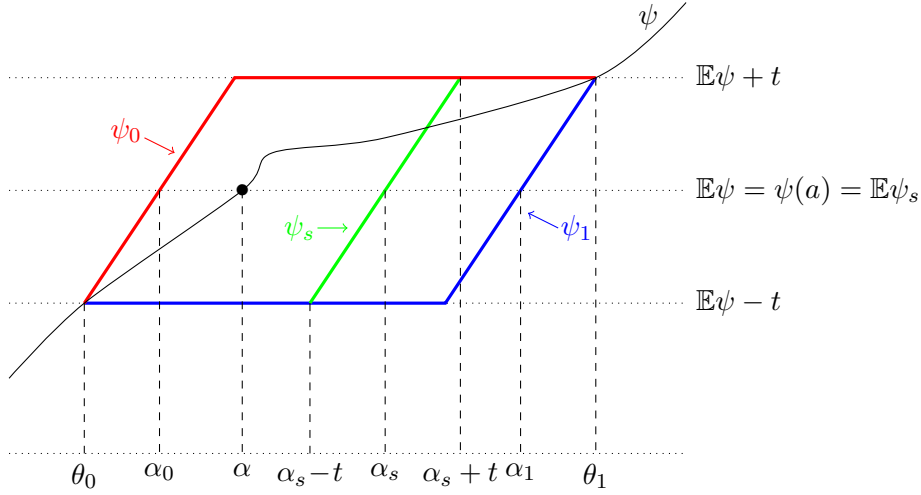
 
\bigskip \noindent Then $\psi_0, \psi_1$ are $1$-Lipschitz, $\psi_0 \geq \psi \geq \psi_1$, and so $\E \psi_0 \geq \E \psi \geq \E \psi_1$. Consequently (again, see Figure \ref{psilambda}), there is an intermediate function $\psi_s$ for some $s \in [0,1]$ such that

\medskip \noindent  
$\bullet$  $\psi_s(\theta) = \E \psi  +\theta-\alpha_s$ \ for \  $\theta \in [\alpha_s-t,\alpha_s +t]$  \\
$\bullet$  $\psi_s(\theta) =  \E \psi -t$ for \  $\theta \in [\theta_0,\alpha_s -t]$ and   $\psi_s(\theta) = \E \psi +t$ for \  $\theta \in [\alpha_s +t, \theta_1]$\\
$\bullet$ $\E \psi_s=\E \psi$

\smallskip \noindent As a consequence of these properties,   the set $\{|\psi_s-\E \psi_s| < t\}=(\alpha_s-t,\alpha_s +t)$ is strictly contained  in the set $\{|\psi-\E \psi| < t\}= (\theta_0,\theta_1)$ and so (again, in all cases of interest) $\nu\big(\{|\psi_s-\E \psi_s| \geq t\}\big) > \nu\big(|\psi-\E \psi | \geq t\big)$. This means that such $\psi$ can not be extremal for the two-sided problem for this particular value of $t$ and shows that an extremal function must satisfy \eqref{linear}. An alternative take on this argument is that we just produced an extremal function, namely $\psi_s$, for which \eqref{linear} holds (this doesn't even require that the inequality stated earlier in this paragraph is strict). 

In the above argument we tacitly assumed that $\theta_0$ and $\theta_1$ existed (i.e., the sets appearing in their definitions were nonempty) and that they belonged to the support of $\nu$; this is what we meant by ``cases of interest.''  However, if -- for example -- the set $\{ \theta : \psi(\theta)  = \E \psi -t \}$  was empty, then it would follow that in fact  $\psi(\theta)  > \E \psi -t$   for all $\theta$ and, consequently, 
$\{| \psi  - \E \psi| \geq t \} = \{ \psi  - \E \psi \geq t \} $. This means that we would be back to the one-sided problem, for which we know that the he functions $\phi_a$ are extremal.  
Another caveat is that if the interval $(\theta_0,\theta_1)$ was not included in $[-\pi/2,\pi/2]$ (the support of the measure $\nu$ from \eqref{def:nu}), it might happen that  $\mu(\{|\psi_s-\E \psi_s| < t\}) = \nu\big( (\alpha_s-t,\alpha_s +t) \big)$  is the same as $\mu(\{|\psi-\E \psi| < t\}) = \nu\big( (\theta_0,\theta_1) \big)$. Again, this is not a problem. First, we can replace one putatively extremal function $\psi$ by another one by modifying it outside of the support of $\nu$, which has no effect on the quantities under consideration. Next,  $(\theta_0,\theta_1) \not\subset [-\pi/2,\pi/2]$ means that we are again {\em de facto} in the setting of the one-sided problem. 

\smallskip \noindent {\it Step $4^\circ$}  To summarize the analysis up to this point, the extremal functions $\psi$  
satisfy the property \eqref{linear}, which can be subsumed as follows:  for some $\alpha \in \R$,  
\be \label{moment} 
\E \psi = 
\psi(\alpha) \quad \hbox{and} \quad \{| \psi  - \E \psi| < t \} = (\alpha -t,\alpha +t). 
\ee
Since the density of $\nu$ decreases away from $0$, it is apparent that $\nu\big((\alpha -t, \alpha +t)\big)$ will be minimized when $|\alpha |$ is as large as possible and, by Step $1^\circ$, it is enough to consider $\alpha \leq 0$.  Given such putatively extremal function $\psi$ (with associated $\alpha$), define   $\widetilde{\psi}$ by 
\be \label{eq:optimal}
\widetilde{\psi}(\theta)= \phi_{\alpha +t}(\theta)  + \psi(\alpha)-\alpha. 
\ee
Then $\widetilde{\psi}(\theta)= \psi(\theta)$ for $\theta \in [a-t,a+t]$  and  $\widetilde{\psi} \leq \psi $ everywhere else.  
Accordingly, $\E \widetilde{\psi}  \leq  \E \psi$, and so if 
$\widetilde{\alpha}$ is defined by $\widetilde{\psi} (\widetilde{\alpha}) = \E \widetilde{\psi}$, then 
$$
\widetilde{\psi} (\widetilde{\alpha}) = \E \widetilde{\psi}  \leq  \E \psi =\psi(\alpha) = \widetilde{\psi} (\alpha). 
$$
It follows that $\widetilde{\alpha} \leq \alpha$ and the inequality is strict unless $\widetilde{\psi}  = \psi$ on the support of $\nu$ (in other words, $\nu$-a.e.).  Consequently,
\be \label{monotone}
\nu\big(\{| \psi  - \E \psi| < t \}\big) = \nu\big((\alpha -t,\alpha +t)\big) \geq \nu\big((\widetilde{\alpha} -t,\widetilde{\alpha} +t)\big) = \nu\big(\{|\widetilde{\psi}  - \E \widetilde{\psi} < t \}\big)
\ee
and so if $\psi$ was extremal, so is $\widetilde{\psi}$. Since the function $\widetilde{\psi}$  is -- up to an additive constant --  of the form $\phi_a$,  \eqref{opt3} follows. 

It remains to justify 
the last assertion of Lemma \ref{extreme2}, namely that ``it is sufficient to consider $t= a-\E \phi_a$.''   
In the argument above, the extremal function $\widetilde{\psi}$ being of the form $\phi_{\alpha +t} +c$,   this condition translates to $\widetilde{\alpha}= \alpha$, an equality which can be immediately deduced in many cases of interest. 
This happens, for example,  if the function $\rho$ defining the density of $\nu$ is strictly decreasing on its support (which holds in our setting, cf.\ \eqref{def:nu}). In that case, the function $\alpha \to  \nu\big((\alpha -t,\alpha +t)\big)$ is strictly increasing for $\alpha < 0$ (excluding, if applicable, the ``trivial'' range, i.e.,  the values of $\alpha$ for which the interval $(\alpha -t,\alpha +t)$ does not intersect the support of $\nu$). Now, if we had $\widetilde{\alpha} < \alpha$, a strict inequality in \eqref{monotone} would follow, contradicting the extremality of $\psi$. 

This special case is sufficient for our intended applications of Lemma \ref{extreme2}. And here is a sketch of the argument addressing the case of general $\rho$.  Let  $\psi_1:= \widetilde{\psi}$ and repeat the construction above with $\psi$ replaced by $\psi_1$ to obtain $\psi_2$ etc. Each $\psi_k$ is extremal and, up to an additive constant, is of the form $\phi_{a_k}$ with $a_1\geq a_2 \geq \ldots$. Next, since $\lim_{x\to \infty} \rho(x) =0$, it follows that 
 $\lim_{\alpha\to -\infty} \nu\big((\alpha -t,\alpha +t)\big) =0$, and so the sequence $(a_k)$ must converge to some finite limit $a$. The limit function $\phi_a$ will be also extremal and will satisfy $t= a-\E \phi_a$. \qed

 \subsection{Proof of Lemma \ref{powers}.} \label{lemma13} 
Since the proof of Lemma \ref{improved} uses the inequalities from Lemma \ref{powers}, we will start with the latter.  Let us recall the main instances of the inequalities in question :  
$$ 
\cos^{n-2}\theta \geq e^{-n\theta^2/2} \quad \hbox{ for } \ \theta \in [0, 3/\sqrt{n}] \ \hbox{ and } \ n \geq 5. \leqno{\eqref{improved2}}
$$
$$ 
\cos^{n-1}\theta \leq e^{-n\theta^2/2} \quad \hbox{ for } \ \theta \in [\sqrt{6/n},\pi/2]  \ \hbox{ and } \ n \geq 3.
\leqno{\eqref{improved4}}
$$
Here is an outline of the proof. We start with \eqref{improved2}, which is effectively a lower bound for the density of $\nu=\nu_n$.  Taking logarithms of both sides and substituting $s = \theta \sqrt{n}$ we see that, for a given $n$,  \eqref{improved2}  is equivalent to. 
\begin{equation} \label{improved3}
h(s)=h_n(s) := ({n-2})\log \cos\frac s{\sqrt{n}} + \frac{s^2}{2}  \geq 0 \quad \hbox{ for } \ s \in [0, 3].
\end{equation}
It is readily verified that $h(0)=h'(0)=0$, while $h''(s) > 0$ on some interval $[0, \alpha)$ (where $\alpha >0$ depends on $n$) and $h''(s) < 0$ for $s>\alpha$.  Now, the domain of $h$ is $[0,\pi\sqrt{n}/2)$  and $\lim_{s\to \pi\sqrt{n}\pi/2} h(s) =-\infty$, so it follows that $h(s) \geq 0$ on some interval $[0, \beta)$ (again, depending on $n$) and $h(s) < 0$ for $s > \beta$. Accordingly, \eqref{improved3}  holds, for a given $n$, iff $h_n(3)\geq 0$. 
Since (by elementary calculus) we have equality in the limit, this would follow once if the sequence 
$\big(h_n(3)\big)$ was nonincreasing. This is not exactly true, but almost: in fact it is nonincreasing starting with $n=7$.  For $n\geq 8$, this can be established numerically by substituting, say,   $s =\frac {3}{\sqrt{n}}$  (to have a compact interval to deal with) and by differentiating with respect to $s$.  The remaining instances are handled by directly evaluating $h_3(2.67), h_4(2.89)$, and $h_k(3)$ for $k=5,6,7$. 

We now pass to the analysis of  \eqref{improved4}.  As in the argument that led to \eqref{improved2}, this reduces to verifying that the inequality holds for $x= \sqrt{6/n}$, and one way to show that is by establishing that the sequence  $\big(\cos^{n-1} \sqrt{6/n}\big)$, which converges to $e^{-3}$, is increasing. Again, this can be verified by the method suggested above. This fact is fairly delicate and is (roughly) equivalent to the inequality $\cos s \leq e^{-3s^2/(6-s^2)}$, which -- while not standard -- may be known, and is also fairly easy to show directly.  \qed

\subsection{Proof of Lemma \ref{improved}.} \label{lemma12}


We first use the bound \eqref{improved2} to prove the estimate from  Lemma \ref{improved} for  $x\in \big[\frac 1{2\sqrt{n}},\sqrt{\frac 6n}\big]$. (Again, the choice of the $\sqrt{\frac 6n}$ cutoff is predicated on the other approach taking care of $x\geq \sqrt{\frac 6n}$.)  The heuristics are as follows. First, having an estimate of the 
 form $e^{-n\theta^2/2}$ allows to express both sides of the inequality in terms of the variable $u= x\sqrt{n}$. Next, it turns out that after this substitution, the density of $\nu$ is bounded from below on the interval $0\leq u \leq 3$ (and, {\em a fortiori}, for $u \in [0,\sqrt{6}]$) by a strictly positive constant, independent of $n$. This means that $u \to \nu_n([u/\sqrt{n}, \pi/2])$ can be upper-bounded on that interval by a strictly decreasing linear function, with equality at $u=0$, and so -- for moderate values of $u$ -- we will have a strict separation between that function and $\frac 12 e^{-nx^2/2}= \frac 12 e^{-u^2/2}$, leading to the asserted improved upper bound. 

Here are some of the details. First, the change of variables $u= x\sqrt{n}, s=\theta\sqrt{n}$ leads to 
\begin{eqnarray*}
\nu_n([0,x]) &=& \big(2I_{n-2}\big)^{-1} \int_0^x \cos^{n-2}\theta\, d\theta\\
&\geq& \big(2I_{n-2}\big)^{-1} \int_0^x  e^{-n\theta^2/2}\, d\theta\\
&=&\big(2I_{n-2}\big)^{-1} \frac 1{\sqrt{n}} \int_0^u  e^{-s^2/2}\, ds .
\end{eqnarray*}
By Proposition \ref{wallis}(ii), the factor in front of the integral converges to $\frac 1{\sqrt{2\pi}}$ as $n\to \infty$ (basically, this reflects the fact that, as we mentioned earlier, the random variable $\sqrt{n}\, \theta$ approximates the standard normal), but -- unfortunately -- is strictly smaller than the limit. If we ignore that discrepancy and compare the ``ideal upper bound'' 
$\frac 12 - \frac 1{\sqrt{2\pi}}\int_0^u  e^{-s^2/2}\, ds$ for $\nu([x,\pi/2])$  to $\frac 25 e^{-u^2/2}$, we get the  picture as shown in Figure \ref{improved-pic}.  

  \begin{figure}[ht!]
\includegraphics[width=0.75\textwidth]{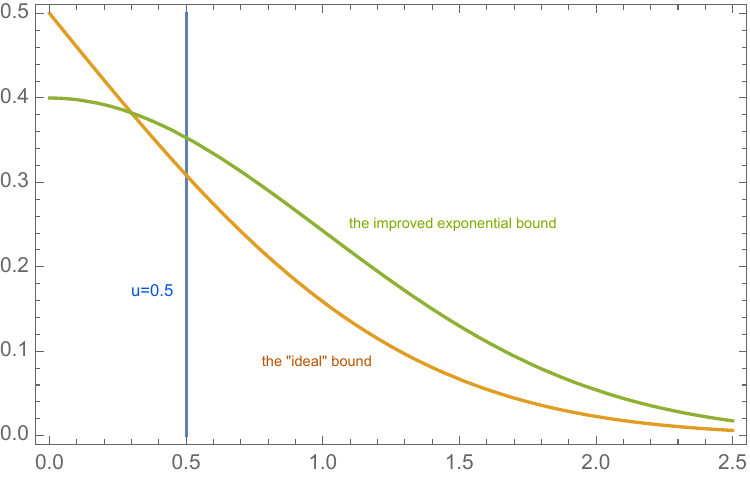}
\caption{The ``ideal upper bound'' for $\nu_n([0,x])$, shown as a function of  $u= x\sqrt{n}$, and the ``improved'' exponential bound $\frac 25 e^{-u^2/2}$. For $u \in [0.5,2.5]$, the former is smaller than the latter, with ``room to spare.''  Note that, for our purposes, we need to analyze  $u \in [0.5,\sqrt{6}]$, but we picture the larger interval $0 \leq u \leq 2.5$ for convenience and clarity.}
\label{improved-pic}
\end{figure}

Clearly and unsurprisingly, there is some room to spare between the ideal upper bound and the bound asserted in Lemma \ref{improved}, which establishes that bound (in the range $\frac 1{2\sqrt{n}} \leq x \leq \sqrt{\frac 6 n}$) for sufficiently large $n$.  Since the constants are explicit, we can verify that ``sufficiently large'' means here ``$n\geq 30$.'' Smaller values of $n$ can be checked directly. 

Finally, we can check directly (numerically) that, for $n=3$, the bound in question is valid when $u=x\sqrt{n} \geq 0.551$. 

\smallskip It remains to show the estimate from  Lemma \ref{improved} for  $x  > \sqrt{\frac {6}{n}}$. In that range, we will use the bound \eqref{cap-cos2} from Proposition \ref{prop-cap}, which restated in the current context asserts that 
\[
\nu_n([x,\pi/2])\leq  (\sqrt{2\pi}\, \kappa_n \sin x)^{-1} \cos^{n-1} x . 
\]
We next appeal to \eqref{improved4} to upper-bound  $\cos^{n-1} x$ by $e^{-nx^2/2}$.   Once this is done, it remains to check that the coefficient $(\sqrt{2\pi}\, \kappa_n \sin x)^{-1}$ doesn't exceed $0.4$  as long as $x\in  \big[ \sqrt{\frac{6}{n}}, \frac{\pi}2\big]$, which is  straightforward and completes the proof.  (The sequence $\big(\frac{\kappa_n }{\sqrt{n}}\big)$ being increasing, see Proposition \ref{wallis}, comes in handy here.) \qed

\subsection{Proof of inequality \eqref{2sidedneg9}} \label{eq60} 
We first rewrite the inequality from the conclusion of \eqref{2sidedneg9} as follows
\be \label{app-neg}
\left(\frac{\cos (b+2t) }{\cos b}\right)^n \geq   \frac{\cos^{2}(b+2t) }{4(n-1)  \sin^2 b} \, .
\ee
By elementary trigonometry, 
\be \label{app-neg2}
\frac{\cos (b+2t) }{\cos b} = 1-2 \sin^2 t -\frac{\sin b}{\cos b} \sin 2t .
\ee
We next use the constraint  $t  \leq \frac{I_{n-1}}{\pi  n \sin^2 b} \times \cos^n b$ from \eqref{2sidedneg9} to upper-bound the last two  terms on the right-hand side of \eqref{app-neg2}. We have 
\begin{eqnarray*}
2 \sin^2 t &\leq & 2 \left(\frac{I_{n-1}}{\pi  n \sin^2 b} \times \cos^n b\right)^2\\
 &\leq &  2 \left(\frac{I_{n-1}}{\pi n \, b^2} \right)^2  \cos^{2(n-1)} b\\
  &\leq & \frac{2I_{n-1}^2}{\pi^2u^4} e^{-u^2} = \big(I_{n-1}^2 n\big) \times  \frac{2 e^{-u^2}}{\pi^2u^4} \times \frac 1n
\end{eqnarray*}
where we used consecutively the inequality $\frac{\cos y}{\sin^2 y} \leq \frac 1{y^2}$,  the substitution $u = b \sqrt{n}$, and the bound \eqref{improved4}, which applies since  $u\geq  \sqrt{6}$.  Similarly 
\begin{eqnarray*}
\frac{\sin b}{\cos b} \sin 2t  &\leq  &\frac{2 I_{n-1}}{\pi  n \sin b} \times \cos^{n-1} b\\
 &\leq & \frac{ I_{n-1}}{ n  b} \times \cos^{n-1} b \\
  &\leq &  \big(I_{n-1} \sqrt{n}\big)\times \frac{e^{-u^2/2}}{u}  \times \frac 1n .
\end{eqnarray*}
We now note that the sequence  $n \to I_{n-1} \sqrt{n}$ decreases to $\sqrt{\frac \pi 2}$ (by Proposition \ref{wallis}) and so, for $n\geq 3$, can be upper-bounded by $I_{2} \sqrt{3} = \frac{\sqrt{3} \pi }{4}$. Moreover, the coefficients  $\frac{2 e^{-u^2}}{\pi^2u^4}$ and $ \frac{e^{-u^2/2}}{u} $ are decreasing functions of $u$, so they can be  upper-bounded by their vales at  $u=\sqrt{6}$. Putting these estimates together, we conclude that
\be \label{app-neg3}
\frac{\cos (b+2t) }{\cos b} \geq 1-\frac{\alpha}n, \quad \hbox{where} \quad \alpha = \frac{1}{96 e^6}+\frac{\pi }{4 \sqrt{2} e^3} < 0.028 .
\ee
At the same time, again for $u\geq \sqrt{6}$ and $n\geq 3$, 
\[
 \frac{\cos^{2}(b+2t) }{4(n-1)  \sin^2 b} \leq  \frac{\cos^{2}b }{4(n-1)  \sin^2 b} \leq \frac n{4(n-1)}\times \frac 1{n b^2} \leq \frac 3{8u^2} \leq \frac 1{16}. 
\]
Now, for $\alpha \in (0,1)$, the sequence  $n \to \left(1-\frac{\alpha}n\right)^n$ increases to $e^{-\alpha}$, so it can be lower-bounded by its initial term. In our setting, the initial term corresponds to  $n=3$, and so 
$\left(1-\frac{\alpha}n\right)^n \geq \left(1-\frac{\alpha}3\right)^3 >0.97$.   This means that the inequality \eqref{app-neg} indeed holds and, again, it is not close. \qed

\vskip1cm 

\end{document}